\documentclass{birkjour}
\usepackage[allcolors=blue,colorlinks=true, pdfstartview=FitH, linkcolor=blue, pdftoolbar=true, bookmarks=true,bookmarksnumbered,plainpages]{hyperref}
\usepackage{graphicx}
\usepackage{caption}
\usepackage{subcaption}
\usepackage[normalem]{ulem}
\usepackage{indentfirst}
\usepackage{tablefootnote}
\usepackage[usenames,dvipsnames,svgnames,table]{xcolor}
\usepackage[slantedGreek]{mathptmx} 
\usepackage{lipsum}
\usepackage{mathtools}
\usepackage[normalem]{ulem}
\usepackage{newtxtext}
\usepackage{enumitem}
\usepackage{amsmath,amssymb}

\theoremstyle{theorem}
\newtheorem{theorem}{Theorem}[section]
\newtheorem{definition}[theorem]{Definition}
\newtheorem{proposition}[theorem]{Proposition}
\newtheorem{corollary}[theorem]{Corollary}
\newtheorem{lemma}[theorem]{Lemma}
\newtheorem{remark}[theorem]{Remark}

\numberwithin{equation}{section}

\newcommand{\R}{\mathbb R}

\newcommand{\N}{\mathbb N}
\newcommand{\Z}{\mathbb Z}

\def\S{{\mathcal S}}

\newcommand{\supp}{\text{supp}\,}

\begin{document}

\title[On the continuity of  strongly singular CZ on $H^p$]{On the continuity of  strongly singular Calder\'on-Zygmund-type operators on Hardy spaces}%

\author {Claudio Vasconcelos}
\address{Departamento de Matem\'atica, Universidade Federal de S\~ao Carlos, S\~ao Carlos, SP, 13565-905, Brazil}
\email{claudio.vasconcelos@estudante.ufscar.br}

\author {Tiago Picon}
\address{Departamento de Computa\c{c}\~ao e Matem\'atica, Universidade S\~ao Paulo, Ribeir\~ao Preto, SP, 14040-901, Brazil}
\email{picon@ffclrp.usp.br}

\thanks{Work supported in part by CAPES, CNPq (grant 311430/2018-0) and FAPESP (grant 18/15484-7)}

\subjclass{42B20; 42B30; 35S05}


\keywords{Calder\'on-Zygmund operators; Hardy spaces; Molecules; Muckenhoupt weight; Pseudodifferential operators; }

\maketitle

\begin{abstract}
	In this work, we establish results on the continuity of strongly singular Calder\'on-Zygmund operators of type $\sigma$  on Hardy spaces $H^p(\R^n)$ for $0<p\leq 1$ assuming a weaker $L^{s}-$type H\"ormander condition on the kernel. Operators of this type include appropriated classes of pseudodifferential operators $OpS^{m}_{\sigma,b}(\R^n)$ and operators associated to standard $\delta$-kernels of type $\sigma$ introduced by \'Alvarez and Milman. As application, we show that strongly singular Calder\'on-Zygmund operators are bounded from $H^{p}_{w}(\R^n)$ to $L^{p}_{w}(\R^n)$, where $w$ belongs to a special class of Muckenhoupt weight.
\end{abstract}

\section{Introduction}

Motivated by the studies of a special class of pseudodifferential operators, J. \'Alvarez and M. Milman in {\cite[Definition 2.1]{AlvarezMilman}} introduced a new class of Calder\'on-Zygmund operators associated to $\delta$-kernels of type $\sigma$, which are continuous functions away the diagonal on $\R^{2n}$ that satisfy
\begin{equation} \label{pontual_kernel_tipo_sigma}
	|K(x,y)-K(x,z)|+|K(y,x)-K(z,x)| \leq C \dfrac{|y-z|^{\delta}}{|x-z|^{n+ \frac{\delta}{\sigma}}},
\end{equation}
for all $|x-z| \geq 2|y-z|^{\sigma}$ some $0<\delta\leq1$ and $0<\sigma <1$. A linear and bounded operator $T: \mathcal{S}(\R^{n}) \rightarrow \mathcal{S}'(\R^{n})$ is called a \textit{strongly singular Calder\'on-Zygmund operator} if it is associated to a $\delta$-kernel of type $\sigma$ in the sense 
$$
\langle Tf,g \rangle = \int{\int{K(x,y)f(y)g(x)dy}dx}, \ \ \forall \ f,g \in \mathcal{S}(\R^n) \mbox{ with disjoint supports,} 
$$
has bounded extension from $L^2(\R^n)$ to itself and in addition $T$ and $T^{\ast}$ extend to a continuous operator from $L^q(\R^n)$ to $L^2(\R^n)$ where 
\begin{equation*}
	\dfrac{1}{q} = \dfrac{1}{2} + \dfrac{\beta}{n} \quad \text{for some} \,\,(1-\sigma) \dfrac{n}{2} \leq \beta < \dfrac{n}{2}.
\end{equation*}

These non-convolution operators are a natural extension of weakly-strongly singular integrals of convolution-type which were studied in the works \cite{Fefferman1970, FeffermanStein1972, Hirschman1959, Wainger1965}. The assumption $L^q-L^2$ continuity is motivated by convolution operators associated to kernels of type $\sigma$ that satisfy the uniform pointwise control  $|\widehat{K}(\xi)|\lesssim\footnote{{The notation $f \lesssim g $ means that there exists a constant $C>0$ such that $f(x)\leq C g(x)$ for all $x \in \R^n$}}(1+|\xi|)^{-\beta}$ for $\beta>0$ (see for instance \cite[Theorem 1]{FeffermanStein1972} with $\beta=n\sigma/2$). Thus, this assumption can be understood  as a suitable correction of the $L^{2}$ continuity due to action of  kernels of type $\sigma$ which are naturally more singular at the diagonal; this justify the nomenclature \textit{strongly singular integral operators}.


A natural question arises on investigating the boundedness on Hardy spaces $H^{p}(\R^n)$ for $0<p\leq 1$  of operators associated to weaker $\delta-$kernels of type $\sigma$ given by H\"ormander condition
\begin{equation} \label{hormander}
	\int_{|x-z| \geq 2|y-z|^{\sigma}}|K(x,y)-K(x,z)|+|K(y,x)-K(z,x)|dx \leq C \ 2^{-\delta}.
\end{equation}
This question for strongly singular type operators is still open in general, moreover it is known that for non-convolution operators and $\sigma=1$ the condition \eqref{hormander} is not sufficient to obtain boundedness on $H^{1}(\R^n)$, even provided $L^{2}(\R^n)$ continuity (see \cite[Theorem 2]{YangYanDeng-Hormandercondition}). Clearly $\delta-$kernels of type $\sigma$ satisfy \eqref{hormander}.    

In this work we continue the program proposed by J. \'Alvarez and M. Milman in \cite{AlvarezMilman} on strongly singular Calder\'on-Zygmund-type operators presenting news perspectives for continuity on Hardy spaces $H^{p}(\R^n)$ for $0<p\leq 1$ assuming a
$L^s-$ type H\"ormander condition of the associated kernel. We say that a kernel $K(x,y)$ associated to $T$ satisfies the $D_{s}$ condition for $s \geq 1$ if
\begin{equation} \label{s-hormander-1}
	\displaystyle{\left( \int_{C_j(z,r)}{|K(x,y)-K(x,z)|^s+|K(y,x)-K(z,x)|^sdx} \right)^{\frac{1}{s}} \lesssim |C_j(z,r)|^{\frac{1}{s}-1} \, 2^{-j\delta}} 
\end{equation}
for $r>1$ and
\begin{align} \label{s-hormander-2}
	&\left( \int_{C_j(z,r^{\rho})}{|K(x,y)-K(x,z)|^s+|K(y,x)-K(z,x)|^sdx} \right)^{\frac{1}{s}} \nonumber  \\  
	& \quad \quad \quad \quad \quad \quad \quad \quad \quad \quad \quad \quad \quad \quad \quad \quad \lesssim |C_j(z,r^{\rho})|^{\frac{1}{s}-1+\frac{\delta}{n} \left(\frac{1}{\rho}-\frac{1}{\sigma}  \right)}  2^{-\frac{j\delta}{ \rho}} 
\end{align}
for $r<1$, where  $C_j(z,\tilde{r})=\{ x \in \R^n: \ 2^j \tilde{r}<|x-z| \leq 2^{j+1}\tilde{r} \}$, $0< \rho \leq \sigma \leq 1$, $z \in \R^n$ and $|y-z|<r$. It is easy to verify that $D_{s_{1}}$ condition is stronger than a $D_{s_{2}}$  if $s_{1}>s_{2}$. We remark that a simple computation shows that $\delta-$kernels of type $\sigma$ satisfy $D_{s}$ conditions for every $1 \leq s<\infty$. {By simplicity, we use the nomenclature $D_s$ condition omitting the dependence of $\delta$. On the other hand, as necessary we emphasize that the kernel satisfies the $D_s$ condition with decay $\delta$ (see for instance Proposition \ref{exemplo_pseudo_2}).}

Estimates of this type are slightly different of $D_{s,\alpha}$ conditions presented in \cite[Definition 1.1]{DeFrancia19867} and they are naturally related to distributional kernels associated to some classes of  pseudodifferential operators in the H\"ormander class $OpS^{m}_{\sigma,b}(\R^n)$ with $0<\sigma \leq 1$ and $0\leq b <1$ (see \cite[Theorem 5.1]{AlvarezHounie}). In particular, {it has been shown that} if $b\leq \sigma$ and $m\leq -n(1-\sigma)/2$, the kernel $K(x,y)$ associated to $T \in OpS^{m}_{\sigma,b}(\R^n)$ satisfies the $D_{1}$ condition i.e., for $r>1$
\begin{equation} \label{s-hormander-1a}
	\displaystyle{\int_{C_j(z,r)}{|K(x,y)-K(x,z)|+|K(y,x)-K(z,x)|dx}  \lesssim \, 2^{-j\delta}, \quad \text{for} \quad r>1,}
\end{equation}
and for $r<1$
\begin{align} \label{s-hormander-2a}
	\int_{C_j(z,r^{\rho})}{|K(x,y)-K(x,z)|+|K(y,x)-K(z,x)|dx} \lesssim \, |C_j(z,r^{\rho})|^{\frac{\delta}{n} \left(\frac{1}{\rho}-\frac{1}{\sigma}  \right)} \, 2^{-\frac{j\delta}{ \rho}}.
\end{align}
{The previous $L^1-$ type H\"ormander conditions  \eqref{s-hormander-1a} and \eqref{s-hormander-2a}
are slightly stronger then \eqref{hormander} and it is still an open question the boundedness on $H^{1}(\R^n)$ from it. Integral estimates such as \eqref{s-hormander-1} and \eqref{s-hormander-2} supply a weaker condition compared to \eqref{pontual_kernel_tipo_sigma}, which has been extensively studied for convolution and non-convolution standard Calder\'on-Zygmund operators (see \cite[p. 315]{GarciaFranciaWeighted} and \cite[p. 23]{MeyerCoifman1997}).}   

Our first result is concerning the continuity of strongly singular Calder\'on-Zygmund operators on $H^p(\R^n)$ satisfying the $D_s$ condition, with emphasis on the relation of the parameters involved (previously omitted for $\delta-$kernels of type $\sigma$). The result is the following:


\begin{theorem} \label{teorema_s-hormander}
	Let $T: \mathcal{S}(\R^{n}) \rightarrow \mathcal{S}'(\R^{n})$ be a bounded linear operator and suppose that:
	\begin{enumerate}
		\item[\textnormal{(i)}] $T$ extends to a continuous operator from $L^2(\R^n)$ to itself;
		\item[\textnormal{(ii)}] $T$ is associated to a kernel satisfying the $D_{s_1}$ condition; 
		\item[\textnormal{(iii)}] For some $n(1-\sigma) \left( 1- \dfrac{1}{s_2} \right) \leq \beta < n\left( 1- \dfrac{1}{s_2} \right)$, $T$ extends to a continuous operator from $L^q(\R^n)$ to $L^{s_2}(\R^n)$, where $\dfrac{1}{q} = \dfrac{1}{s_2} + \dfrac{\beta}{n}$ and $s_{2}>1$.
	\end{enumerate}
	Under such conditions, if $T^{\ast}(x^{\alpha})=0$ for all $\alpha \in \Z^{n}_{+}$ such that $|\alpha| \leq \lfloor \delta \rfloor$, $1< s_1 \leq 2$ and $s_1 \leq s_2$, then the operator $T$ can be extended to a bounded operator from $H^p(\R^n)$ to itself for $p_{_0}<p \leq 1$, where
	\begin{equation}\label{pcritico}
		\frac{1}{p_{_0}}:= \dfrac{1}{s_2}+\dfrac{\beta \left[ \dfrac{\delta}{\sigma}+n\left(1-\dfrac{1}{s_2}\right)\right]}{n\left(\dfrac{\delta}{\sigma}-\delta+\beta\right)} \ .
	\end{equation}
	Moreover, if $T^{*}$ also satisfies (iii) then the conclusion holds for $1<s_{1}<\infty$ and $s_{1} \leq s_{2}$. The case $s_{1}=1$ also holds, however only for the range $p_{_0}<p<1$.
\end{theorem}

{This result extends \cite[Theorem 2.2]{AlvarezMilman} with additional advantage of considering boundedness of operators of type $\sigma$ on $H^p(\R^n)$ associated to kernels with weaker integral conditions. In addition, our approach enables us to include the $D_{1}$ condition for $p<1$, which represents the closest of H\"ormander condition we can get so far. In contrast to condition \eqref{pontual_kernel_tipo_sigma}, although any upper bound on $\delta$ is assumed on \eqref{s-hormander-1} and  \eqref{s-hormander-2}, naturally examples with $\delta>1$ are attended with suitable refinement of weaker integral conditions incorporating derivatives of the kernel (see Section \ref{section-derivative-conditions}). An application of Theorem \ref{teorema_s-hormander} is presented at Proposition \ref{exemplo_pseudo_2}. A particular weaker class of kernels of type $\sigma$ is presented at Section \ref{section-yabuta-operators}.}


The proof of Theorem \ref{teorema_s-hormander} follows by showing that $T$ maps $(p,\infty)$-atoms into an appropriate notion of  molecules associated to \eqref{s-hormander-1} and \eqref{s-hormander-2}. The molecules will be presented  in Section \ref{section-molecules}. 

The assumption on the $L^q-L^{s}$ boundedness of $T^{*}$ in the Theorem \ref{teorema_s-hormander} guarantees that $T$ is continuous from $L^\infty(\R^n)$ to $BMO(\R^n)$.

\begin{theorem} \label{theorem-Linf-BMO}
	Let $T$ be an operator satisfying (i) and (iii) as in Theorem \ref{teorema_s-hormander} with kernel satisfying the $D_1$ condition. Assume also that (iii) holds for $T^{\ast}$. Then $T$ is continuous from $L^{\infty}(\R^n)$ to $BMO(\R^n)$.
\end{theorem} 

The statement is analogous to \cite[Theorem 2.1]{AlvarezMilman}  announced for kernels satisfying  \eqref{pontual_kernel_tipo_sigma}.
The assumption on $T^{*}$ satisfying (iii) can be weakened (see Section \ref{comment}) in the spirit of
\cite[Corollary 3.3]{AlvarezMilmanVectorValued} for vector valued operators assuming the $D_{1,\alpha}-$ condition   \cite[Definition 1.1]{DeFrancia19867}).
The next remark will be fundamental in the complement of Theorem \ref{teorema_s-hormander}.

\begin{remark}\label{1.2.a}
	As a direct consequence of Theorem \ref{theorem-Linf-BMO} and real Interpolation Theorem (see \cite[Theorem 6.8]{JavierFourier}), $T$ is a bounded operator from $L^p(\R^n)$ to itself for $2\leq p <\infty$.    
\end{remark}

The critical case $p=p_{_0}$ at Theorem  \ref{teorema_s-hormander} remains open, however, the conclusion continues to hold if we replace the target space by $L^{p}(\R^n)$. Our third result is a weighted continuity version for a special class of Muckenhoupt weight $A_{1}$.

\begin{theorem}\label{th5.2}
	Let $T: \mathcal{S}(\R^{n}) \rightarrow \mathcal{S}'(\R^{n})$ be a linear and bounded operator as in Theorem \ref{teorema_s-hormander}. Then, $T$ can be extended to a bounded operator from $H_{w}^{p}(\R^n)$ to $L_{w}^{p}(\R^n)$  for $p_{_0} \leq p\leq 1$, with $p_{_0}$ given by \eqref{pcritico}, provided $w \in A_1 \cap RH_{d}$ in which $d=\max\left\{\frac{s}{s-p}, \, \frac{s_1}{p(s_1-1)}\right\}$ and $s=\min \left\{2,s_{2} \right\}$. Moreover, if $s_{1}=1$ the conclusion holds for $p_{_0}\leq p<1$ and $d=\frac{1}{1-p}$.
\end{theorem}
In particular the previous result covers the \cite[Theorem 2]{LiLu2006}  due to J. Li and S. Lu taking $s_1=s_2=2$, since  $\delta-$kernels of type $\sigma$ satisfy the $D_{2}$ condition.

The organization of the paper is as follows. In Section \ref{preli} we present general aspects
of Hardy spaces and some properties on $T$ associated to kernels satisfying the $D_s$ condition,  in particular the well definition of $T^{*}(x^{\alpha})$ for $|\alpha|\leq \lfloor \delta \rfloor$. The Section \ref{section-molecules} is devoted to present the appropriated molecules that will be used in the proof of Theorem  \ref{teorema_s-hormander}, in special an atomic decomposition result for these molecules (see Theorem  \ref{molecular_s_decomposition}). In Section \ref{secfour}  we present the proof of Theorems  \ref{teorema_s-hormander}  and \ref{theorem-Linf-BMO}, including some comments and remarks. Also, in Section \ref{section-derivative-conditions} we provide a version of Theorem \ref{teorema_s-hormander} assuming some  weaker integral derivative conditions. In Section \ref{section_aplication}, we recall some basic definitions on weights considered at Theorem \ref{th5.2} and its proof. Lastly, we show that pseudodifferential operators in the H\"ormander class $OpS^{\,m}_{\sigma, b}(\R^n)$ where  $b \leq \sigma$ and $m \leq -n(1-\sigma)/2$ satisfies the $D_{s}$ integral derivative condition, in particular also satisfies \eqref{s-hormander-1} and \eqref{s-hormander-2}, for all $1 \leq s\leq 2$, extending the \cite[Theorem 2.1]{AlvarezHounie}. In Section \ref{section-yabuta-operators}, we provide some generalizations and extensions to operators satisfying Dini-type modulus of continuity.

\section{Preliminaries}\label{preli}

Let $\varphi \in \S(\R^n)$ be such that $\int_{{\R^n}}{\varphi(x)dx} \neq 0$. We define the maximal operator by
$$
\mathcal{M}_{\varphi}f(x):=\sup_{t>0}{|f \ast \varphi_t(x)|}
$$
where $\varphi_t(x)=t^{-n} \varphi(x/t)$. We say that a tempered distribution $f$ belongs to the Hardy space $H^p(\R^n)$ for $p>0$ if there exists $\varphi$ as before such that $\mathcal{M}_{\varphi}f \in L^p(\R^n)$. The functional $\|u\|_{H^p}$ defines a quasi-norm for $0<p<1$ and a norm for $p \geq1$ (we refer as a norm for $0<p<\infty$ by simplicity). In particular $H^p(\R^n)=L^p(\R^n)$ for $p>1$ with equivalent norms and $H^1(\R^n) \subsetneq L^1(\R^n)$ with continuous inclusion. Moreover  $H^p(\R^n)$ is a complete metric space with the distance $d(u,v)=\|u-v\|_{H^p}^p$ with  $u,v\in  H^p(\R^{n})$ for $0<p \leq 1$.   

For $0<p\leq 1$ the space $H^p(\R^n)$ can be decomposed into special functions called atoms, that we present in the sequence. 

\begin{definition}
	Let $0<p\leq1$ and $1\leq s \leq \infty$ with $p<s$. We say that a measurable function $a(x)$ is a $(p,s)-$atom in $H^p(\R^n)$ if there exist a ball $B:=B(z,r)\subset \R^n$ such that
	\begin{equation*}
		\textnormal{(i)} \ \supp(a) \subset B; \quad \textnormal{(ii)} \ \| a \|_{L^s} \leq |B|^{\frac{1}{s}-\frac{1}{p}}; \quad \textnormal(iii) \ \int{a(x)x^{\alpha}}dx=0 \ \textnormal{for all} \ |\alpha|\leq N_p
	\end{equation*}
	where $N_p:= \lfloor n(1/p-1) \rfloor$. For the limit case $s=\infty$, condition (ii) is $\| a \|_{L^\infty} \leq |B|^{-1/p}$.
\end{definition}

\begin{theorem}\cite[Theorem 4.10]{GarciaFranciaWeighted}
	Let $0<p\leq 1$, $1\leq s \leq \infty$, $p<s$ and $f \in H^p(\R^n)$. Then, there exist a sequence $\{a_j\}_{j \in \N}$ of $(p,s)-$atoms and complex scalars $\{\lambda_j\}_{j \in \N}$ such that in $H^p(\R^n)$ we have
	$
	f=\sum_{j=0}^{\infty}{\lambda_ja_j}
	$
	with $\| f \|_{H^p} \approx \text{inf}\,\left( \sum{|\lambda_j|^p} \right)^{1/p}$.
\end{theorem}

\begin{definition}\cite[p. 23]{MeyerCoifman1997}
	Let $m \in \mathbb{N} \cup \{0\}$. We say that an operator $T$ satisfies $T^{*}(x^{\alpha})=0$ for $|\alpha|\leq m$ if
	\begin{equation} \label{meyer}
		\int_{\R^{n}}x^{\alpha}Ta(x)dx=0, \quad \forall \ a \in L^{2}_{\#,m}(\R^{n})
	\end{equation}
	where $L^{2}_{\#,m}(\R^{n})$ is the set of functions in $L_{c}^{2}(\R^{n})$ such that $\int_{\R^{n}}x^{\alpha}f(x)dx=0$ for all $|\alpha|\leq m$. Here $L_{c}^{2}(\R^{n})$ is the set of functions in $L^{2}(\R^{n})$ with compact support.  
\end{definition}

\begin{remark}
	If $a(x)$ is a $(p,\infty)-$atom, then $a \in L^{2}_{\#,m}(\R^{n})$ with $m:=N_{p}$.
\end{remark}

Next we show that $T^{\ast}(x^{\alpha})$ is well defined where $T$ is given by Theorems \ref{teorema_s-hormander} and \ref{theorem-Linf-BMO}.

\begin{proposition} \label{estimative-L1-moments}
	Let $T$ be a linear and bounded operator on $L^2(\R^n)$ whose kernel associated to it satisfies the $D_1$ condition. Then $x^{\alpha}Ta(x) \in L^1(\R^n)$ for all $a \in L^{2}_{\#,m}(\R^n)$ with $|\alpha| \leq m < \delta$.
\end{proposition}

\begin{proof}
	By simplicity suppose that $\supp(a) \subset B(0,r)$ and  $r \geq 1$. Write
	$$
	\int_{\R^n}{|x^{\alpha}Ta(x)|dx} = \int_{B(0,2r)}{|x^{\alpha}Ta(x)|dx}+\int_{\R^n \setminus B(0,2r)}{|x^{\alpha}Ta(x)|dx}.
	$$
	From H\"older inequality and boundedness of $T$ on $L^2(\R^n)$ we get
	\begin{align*}
		\int_{B(0,2r)}{|x^{\alpha}Ta(x)|dx}  \leq  \| x^{\alpha} \|_{L^{\infty}(B(0,2r))} \ |B(0,2r)|^{\frac{1}{2}} \ \| Ta \|_{2} 
		\lesssim   r^{|\alpha|+\frac{n}{2}}  \| a \|_{2} <\infty.
	\end{align*}
	For the second integral, since $a \in L^{2}_{\#,m}(\R^n)$ we may estimate
	\begin{align*}
		|Ta(x)| = \left| \int_{B(0,r)}{[K(x,y)-K(x,0)]a(y)dy} \right|  \leq \int_{B(0,r)}{|K(x,y)-K(x,0)||a(y)|dy}.
	\end{align*}
	Then
	\begin{align*}
		\int_{\R^n \setminus B(0,2r)}{|x^{\alpha}Ta(x)|dx} 
		& \leq  \int_{B(0,r)}{|a(y)| \int_{\R^n \setminus B(0,2r)}{|x|^{|\alpha|} \ |K(x,y)-K(x,0)|dx}   dy} \nonumber \\
		& \leq  \sum_{j=0}^{\infty}{{(2^jr)^{|\alpha|}}\int_{B(0,r)}{|a(y)| \int_{C_j(0,r)}{ |K(x,y)-K(x,0)|dx}   dy}} \nonumber \\
		& \lesssim   \sum_{j=0}^{\infty} (2^jr)^{|\alpha|}  \|a\|_{2} \ |B(0,r)|^{\frac{1}{2}} \  2^{-j\delta} \\
		& \leq  r^{|\alpha|+\frac{n}{2}} \ \| a \|_{2} \  \sum_{j=0}^{\infty}{(2^j)^{|\alpha|-\delta}}<\infty
	\end{align*}
	since $|\alpha| < \delta$. For $r<1$ we may write
	$$
	\int_{\R^n}{|x^{\alpha}Ta(x)|dx} = \int_{B(0,2r^{\rho})}{|x^{\alpha}Ta(x)|dx}+\int_{\R^n \setminus B(0,2r^{\rho})}{|x^{\alpha}Ta(x)|dx},
	$$
	for some $0<\rho\leq \sigma <1$. The estimate for the first integral is the same as the previous case and for the second
	\begin{align}
		\int_{\R^n \setminus B(0,2r^{\rho})}{|x^{\alpha}Ta(x)|dx} & \leq \sum_{j=0}^{\infty}{(2^jr^{\rho})^{|\alpha|}\int_{B(0,r)}{|a(y)|  \int_{C_j(0,r^{\rho})}{|K(x,y)-K(x,0)|dx}  dy }} \nonumber \\
		& \lesssim  r^{|\alpha| \rho +\frac{n}{2} + \delta - \frac{\rho \delta}{\sigma}} \ \| a \|_{2} \ \sum_{j=0}^{\infty}{(2^j)^{|\alpha|-\frac{\delta}{\sigma}}}<\infty \nonumber 
	\end{align}
	since $|\alpha| < \delta$ that implies $|\alpha| < {\delta}/{\sigma}$.
\end{proof}

\section{Generalized Molecules} \label{section-molecules}

M. Taibleson and G. Weiss introduced in \cite{Taibleson-Weiss} the molecular structure in $H^p(\R^n)$ and a useful method to establish continuity on Hardy spaces for certain classes of linear operators by showing that atoms are mapped into molecules i.e., special mensurable functions on $H^p(\R^n)$ with peculiar control in $L^{s}-$norm \cite[p. 71]{Taibleson-Weiss}. In particular, this is used to show continuity of standard Calder\'on-Zygmund operators (see \cite[Theorem 1.1]{AlvarezMilman}). However, for operators associated to $\delta-$ kernels of type $\sigma$ the classical definition of molecules does not apply in an effective way for this purpose and a notion of molecules that best fit these types of operators is necessary. Such ideas were originally explored by B. Bordin \cite{Bordin}, J. \'Alvarez and M. Milman \cite{AlvarezMilman}. 

Next we define a slight generalization of molecules presented in \cite[Definition 2.2]{AlvarezMilman}.

\begin{definition} \label{s-molecule}
	Let $0<p, \rho <1<q<s<\infty$, $1\leq t <\infty$ and $t\leq s$ such that
	$$
	n\left(\frac{t}{p}-1\right)< \lambda \leq n \left( \frac{t}{s}-1 \right)+ \frac{n\,t}{1-\rho}\left(\frac{1}{q}-\frac{1}{s}\right).
	$$
	If $p=1$, we restrict $1<t<\infty$. We say that a function $M(x)$ is a $(p, \rho,  q,  \lambda, s,t)-$molecule if there exist a ball $B(z,r) \subset \R^n$ and a constant $C>0$ such that for $r>1$
	\begin{enumerate}
		\item[\textnormal{M1.}] $\displaystyle \int_{\R^n}{|M(x)|^{t}dx} \leq C \ r^{n \left( 1-\frac{t}{p} \right)}$;
		\item[\textnormal{M2.}] $\displaystyle \int_{\R^n}{|M(x)|^{t}|x-z|^{\lambda}dx \leq C \ r^{\lambda + n \left( 1-\frac{t}{p} \right)}}$,
	\end{enumerate}
	and for $r\leq 1$
	\begin{enumerate}
		\item[\textnormal{M3.}] $\displaystyle \int_{\R^n}{|M(x)|^{t}dx} \leq C \ r^{n \left[\rho \left( 1-\frac{t}{s} \right)+ t \left(\frac{1}{q}-\frac{1}{p}\right) \right]}$;
		\item[\textnormal{M4.}] $\displaystyle \int_{\R^n}{|M(x)|^{t}|x-z|^{\lambda}dx \leq C \ r^{\, \rho \lambda + n \left[\rho \left( 1-\frac{t}{s} \right)+ t \left( \frac{1}{q}-\frac{1}{p} \right) \right]}}$.
	\end{enumerate}
	Besides that, $M(x)$ satisfies for all $r>0$
	\begin{enumerate}
		\item[\textnormal{M5.}] $\displaystyle \int_{\R^n}{M(x)x^{\alpha}dx}=0 \quad \forall \,\,|\alpha| \leq N_p$.
	\end{enumerate}
\end{definition}

\begin{remark} \label{remark-molecules}
	\begin{itemize}
		\item[\textnormal{(a)}] \textnormal{If $t=s=2$ we recover the molecules defined by \cite[Definition 2.2]{AlvarezMilman};}
		\item[\textnormal{(b)}] \textnormal{In order to show that $M(x)$ satisfies conditions (M1) and (M2) it is sufficient verify simultaneously} 
		\begin{enumerate}
			\item[\textnormal{M1a.}] $\displaystyle{ \int_{B(z,r)}{|M(x)|^{t}dx} \leq C \ r^{n \left( 1-\frac{t}{p} \right)}}$;
			\item[\textnormal{M2a.}] $\displaystyle{ \int_{\R^n \backslash B(z,r)}{|M(x)|^{t}|x-z|^{\lambda}dx \leq C \ r^{\lambda + n \left( 1-\frac{t}{p} \right)}}}.$
		\end{enumerate}
		\textnormal{Similar idea applies to (M3) and (M4) integrating on $B(z,r^{\rho})$ and $\R^n \setminus B(z,r^{\rho})$ respectively.}
	\end{itemize}
\end{remark}

Next we verify that condition (M5) is well defined for $M(x)$ satisfying (M1) to (M4).

\begin{proposition} \label{molecule-local-integrable}
	Suppose $M(x)$ is a function satisfying (M1) and (M2) or (M3) and (M4). Then $x^{\alpha}M(x)$ is an absolutely integrable function.
\end{proposition}

\begin{proof}
	{Suppose $1<t<\infty$.} Let $M$ be a function satisfying (M3) and (M4) for $r\leq1$. Split
	$$
	\int_{\R^n}{|x^{\alpha}M(x)|dx}=\int_{B(z,r)}{|x^{\alpha}M(x)|dx}+\int_{\R^n \setminus B(z,r)}{|x^{\alpha}M(x)|dx}.
	$$
	For the first integral, from H\"older inequality and (M3) we have
	\begin{align*}
		\int_{B(z,r)}{|x^{\alpha}M(x)|dx} &\leq \| x^{\alpha} \|_{L^{\infty}(B(z,r))} \, |B(z,r)|^{1-\frac{1}{t}} \, \left(\int_{\R^n}{|M(x)|^tdx}\right)^{\frac{1}{t}} <\infty
	\end{align*}
	and for the second 
	\begin{align*}
		\int_{\R^n \setminus B(z,r)}|& x^{\alpha}M(x)|dx \leq \sum_{|\gamma| \leq |\alpha|} C_{_{|\alpha|, |\gamma|}} \, |z|^{|\alpha|-|\gamma|} \int_{\R^n \setminus B(z,r)}{|M(x)| \, |x-z|^{\frac{\lambda}{t}+\left(|\gamma|-\frac{\lambda}{t}\right)}dx} \\	
		&\leq \| M \, |\cdot -z|^{\frac{\lambda}{t}} \|_{L^t} \,\sum_{|\gamma| \leq |\alpha|} C_{_{|\alpha|, |\gamma|, |z|}} \left( \int_{\R^n \setminus B(z,r)}{|x-z|^{\left( |\gamma|-\frac{\lambda}{t} \right)\frac{t}{t-1}}} \right)^{1-\frac{1}{t}}<\infty	
	\end{align*}
	where the integrability of $|x-z|^{ \left( |\gamma|-\frac{ \lambda}{t} \right)\frac{t}{t-1}}$ on $\R^n \backslash B(z,r)$ is guaranteed by $\lambda>n \left( t/p-1 \right)$ and $|\gamma|\leq |\alpha| \leq n(1/p-1)$. The estimate assuming (M1) and (M2) for $r>1$ follows analogously {as well as the case $t=1$.}
\end{proof}

\begin{lemma} \label{molecular_s_decomposition}
	Let $M(x)$ be a $(p,\rho, q, \lambda,s,t)-$molecule. Then 
	$
	M = \sum_{j=0}^{\infty}{\gamma_j a_j}
	$
	in  $\mathcal{S}'(\R^n)$ 
	where $a_j(x)$ is a $(p,t)-$atom and $\{ \gamma_j \}_{j \in \N}$ a sequence of complex scalars such that $\sum_{j}{|\gamma_j|^p}< \infty$. In particular, there exists a constant $C>0$ independent of $M$ such that $\|M\|_{H^{p}}\leq C$.
\end{lemma}

\begin{proof}
	This proof is inspired in \cite[Theorem 7.16]{GarciaFranciaWeighted} and \cite[Lemma 2.1]{AlvarezMilman}. Let $M(x)$ be a $(p,\rho,q , \lambda,s,t)-$molecule associated to a ball $B=B(z,r)\subset \R^n$ and define $B_j=B(z,2^jr)$, $E_j=B_j \setminus B_{j-1}$ and $M_j(x)=M(x) \chi_{_{E_j}}(x)$ for each $j \in \N \cup \{0\}$. Consider $\alpha \in \Z^{n}_{+}$ a multi-index such that $|\alpha| \leq N_p$, $P_{N}$ the finite-dimensional vector space of polynomials in $\R^n$ with degree at most $N_p$ and $P_{N,j}$ its restriction on the set $E_j$. By the Gram-Schmidt orthogonalization process on the Hilbert space $\mathcal{H}:=L^{2}(E_{j},|E_{j}|^{-1}dx)$, considering $P_{N,j}$ as a subspace of $\mathcal{H}$, with respect to the base $\{ x^{\beta} \}_{|\beta|\leq N_p}$ there exist polynomials $\phi_{\gamma}^{j}(x)$ defined on $P_{N,j}$ uniquely determined such that
	\begin{equation} \label{AF-2}
		\dfrac{1}{|E_j|} \int_{E_j}{\phi_{\gamma}^{j}(x) \  x^{\beta}dx}=\delta_{\gamma,\beta}=\left\{\begin{array}{rc}
			1 &\mbox{if}\quad \gamma=\beta \\
			0 &\mbox{if}\quad \gamma \neq \beta.
		\end{array}\right.
	\end{equation}
	In addition, these polynomials satisfies the estimate $(2^jr)^{|\gamma|} |\phi_{\gamma}^{j}(x)| \leq C,\,\, \forall \, x \in E_{j}$, uniformly on $j$ ({see \cite[p. 332]{GarciaFranciaWeighted}}). Set
	$$
	M_{j}^{\gamma}=\frac{1}{|E_j|} \int_{E_j}{M(x)x^{\gamma}dx}, \quad P_j(x)=\sum_{|\gamma| \leq N_p}{M_{j}^{\gamma}\phi_{\gamma}^{j}(x)},
	$$ 
	and  split
	$$
	M =  \sum_{j=0}^{\infty}{M_j} = \sum_{j=0}^{\infty}{(M_j-P_j)}+ \sum_{j=0}^{\infty}{P_j}
	$$
	where the convergence is given in $L^t(\R^n)$. 	We will show that $(M_j-P_j)$ is a multiple of a $(p,t)-$atom and $P_j$ can be written as a finite linear combination of $(p,\infty)-$atoms for each $j$.
	
	Suppose first $r \leq 1$. Since $M_j$ and $P_j$ are both supported in $E_j \subset B_j$ clearly $\supp(M_j-P_j) \subset B_j$ and also for all $|\alpha| \leq N_p$ it has the right cancellation property from
	\begin{align}
		&\int_{\R^n}\left( M_j(x)-P_j(x)\right) x^{\alpha}dx = \int_{E_j}{\left( M_j(x) - \sum_{|\gamma| \leq N_p}{M_{\gamma}^{j} \phi_{\gamma}^{j}(x)} \right)x^{\alpha}dx} \nonumber \\
		& \quad =  \int_{E_j}{M(x)x^{\alpha}dx}-\sum_{|\gamma|\leq N_p}{\left( \int_{E_j}{M(y)y^{\gamma}dy} \right)\frac{1}{|E_j|} \int_{E_j}{\phi_{\gamma}^{j}(x)x^{\alpha}dx}} = 0 \label{AF-4}
	\end{align}
	In order to estimate the $L^t-$norm of $M_{j}$, follows from condition (M4) that
	{\begin{align}
			\| M_j \|_{L^t} & \leq (2^jr)^{-\frac{\lambda}{t}} \left( \int_{E_j}{|M(x)|^t|x-z|^\lambda dx}\right)^{\frac{1}{t}} \nonumber \\
			& \lesssim \, |B_j|^{\frac{1}{t}-\frac{1}{p}} \,\, (2^j)^{-\frac{\lambda}{t}+n \left( \frac{1}{p}-\frac{1}{t} \right)} \,\, r^{-\frac{\lambda}{t}(1-\rho) + n \left[ \rho \left( \frac{1}{t}-\frac{1}{s} \right)+ \frac{1}{q}-\frac{1}{t} \right]} \nonumber \\
			& \lesssim  \, |B_j|^{\frac{1}{t}-\frac{1}{p}} \,\, (2^j)^{-\frac{\lambda}{t}+n \left( \frac{1}{p}-\frac{1}{t} \right)}  \label{M-j-estimation} 
	\end{align}}
	since {$\lambda \leq n \left( \frac{t}{s}-1 \right)+ \frac{n\,t}{1-\rho}\left(\frac{1}{q}-\frac{1}{s}\right)$}. On the other hand
	\begin{align}
		|P_j(x)|  &\leq  \sum_{|\gamma|\leq d} |\phi_{\gamma}^{j}(x)|  {\dfrac{1}{|E_j|}\int_{E_j}{|M(y)| \,\, |y|^{|\gamma|}} \,\,}dy \nonumber \\
		& \leq  \left( \sum_{|\gamma|\leq d}{(2^jr)^{|\gamma|}| \phi_{\gamma}^{j}(x)|} \right) \frac{1}{|E_j|} \int_{E_j}{|M(y)|dy} \nonumber \\
		& \lesssim  \frac{1}{|E_j|} \int_{E_j}{|M_j(y)|dy} \nonumber \\
		& \leq  |E_j|^{-\frac{1}{t}} \ \| M_j \|_{L^t}. \label{AF-5}
	\end{align}
	From estimates \eqref{M-j-estimation} and \eqref{AF-5} we obtain
	$$
	\| M_j-P_j \|_{L^t} \leq 2 \|M_j\|_{L^t} \lesssim |B_j|^{\frac{1}{t}-\frac{1}{p}} \,\, (2^j)^{-\frac{\lambda}{t}+n \left( \frac{1}{p}-\frac{1}{t} \right)}.
	$$
	Finally, write $(M_j-P_j)(x) = d_j \, A_j(x)$ where $d_j = \| M_j-P_j \|_{L^t} \, |B_j|^{\frac{1}{p}-\frac{1}{t}}$ and 
	$$
	A_j(x)=\dfrac{M_j(x)-P_j(x)}{\| M_j-P_j \|_{L^t}} \, |B_j|^{\frac{1}{t}-\frac{1}{p}}.
	$$
	By the previous considerations it is clear that $A_j$ is a $(p,t)-$atom and in addition
	\begin{align}
		\sum_{j=0}^{\infty}{|d_j|^p} = \sum_{j=0}^{\infty}{\|M_j-P_j\|_{L^t}^{p} \, |B_j|^{1-\frac{p}{t}}}
		& \leq  \sum_{j=0}^{\infty}{\left[ |B_j|^{\frac{1}{t}-\frac{1}{p}} \, (2^j)^{-\frac{\lambda}{t}+n \left( \frac{1}{p}-\frac{1}{t} \right)} \right]^p |B_j|^{1-\frac{p}{t}}} \nonumber \\
		& = \sum_{j=0}^{\infty}{(2^j)^{-\frac{\lambda p}{t}+n \left( 1-\frac{p}{t} \right)}} < \infty \nonumber 
	\end{align}
	since $\lambda> n \left( {t}/{p}-1 \right)$.
	
	We show now that $P_j$ is a finite linear combination of $(p,\infty)-$atoms. Define for each $j \in \N \cup \{ 0 \}$
	$$
	N_{\gamma}^{j} :=|E_k|\sum_{k=j}^{\infty}{M^{k}_{\gamma}}=\sum_{k=j}^{\infty}{\int_{E_k}{M(x)x^{\gamma}dx}}
	$$
	and
	$$
	\psi_{\gamma}^{j}(x):=N_{\gamma}^{j+1} \left[ |E_{j+1}|^{-1} \phi_{\gamma}^{j+1}(x)-|E_j|^{-1}\phi_{\gamma}^{j}(x) \right].
	$$
	Then, we can represent $P_j(x)$ in the following way
	\begin{align*}
		\sum_{j=0}^{\infty}{P_j(x)} & = \sum_{j=0}^{\infty}{ \sum_{|\gamma|\leq N_p}{(M_{\gamma}^{j}|E_j|)\,(|E_j|^{-1}\phi_{\gamma}^{j}(x))}} \\
		& = \sum_{|\gamma|\leq N_p}{\left\{ \sum_{j=0}^{\infty}{\psi_{\gamma}^j(x)} + \sum_{j=0}^{\infty}{\left[ N_{\gamma}^{j}|E_j|^{-1}\phi_{\gamma}^j(x)-N_{\gamma}^{j+1}|E_{j+1}|^{-1}\phi_{\gamma}^{j+1}(x) \right]} \right\}} \nonumber \\
		& = \sum_{j=0}^{\infty}{\sum_{|\gamma|\leq N_p}{\psi_{\gamma}^{j}(x)}}, \nonumber
	\end{align*}
	since 
	$$
	\sum_{j=0}^{\infty}{\left[ N_{\gamma}^{j}|E_j|^{-1}\phi_{\gamma}^j(x)-N_{\gamma}^{j+1}|E_{j+1}|^{-1}\phi_{\gamma}^{j+1}(x) \right]}=N_{\gamma}^{0}=\int_{\R^n}{M(x)x^{\gamma}dx}=0
	$$ 
	for all $|\gamma|\leq N_p$. We claim that $\psi_{\gamma}^j(x)$ is a multiple of a $(p,\infty)-$atom. 
	By definition $\supp(\psi_{\gamma}^j) \subset E_j \subset B_j $ and moreover 
	\begin{equation} \label{AF-07}
		\int _{\mathbb{R}^n}\psi_{\gamma}^j(x)x^{\beta}dx = N_{\gamma}^{j+1}\left[\frac{1}{|E_{j+1}|}\int_{E_{j+1}}\phi_{\gamma}^{j+1}(x)x^{\beta}dx-\frac{1}{|E_{j}|}\int_{E_{j}}\phi_{\gamma}^{j}(x)x^{\beta}dx\right]=0
	\end{equation}
all $|\beta| \leq N_p$. In order to control $ \| \psi_{\gamma}^{j} \|_{L^\infty}$ it is enough to estimate $N_{\gamma}^{j+1}$.
	Then, by H\"older inequality and \eqref{M-j-estimation}  we have
	\begin{align*} 
		|N_{\gamma}^{j+1}|  =  \left|\sum_{k=j+1}^{\infty}{ \int_{E_k}{M(x)x^{\gamma}dx}} \right| 
		&\leq \sum_{k=j+1}^{\infty}{(2^kr)^{|\gamma|}\int_{E_k}{|M_k(x)|dx}} \\
		&\leq  \sum_{k=j+1}^{\infty}{(2^kr)^{|\gamma|} \ \| M_k \|_{L^t} \ |E_k|^{1-\frac{1}{t}}}\\
		& \lesssim  {r^{|\gamma|+n \left( 1-\frac{1}{p} \right)} \sum_{k=j+1}^{\infty}  (2^k)^{|\gamma|-\frac{\lambda}{t}+n\left(1-\frac{1}{t}\right)}} .
	\end{align*}
	Since $|\gamma|\leq n \left( {1}/{p}-1 \right)$ and $\lambda> n \left( {t}/{p}-1 \right)$ the series above converges and	
	\begin{align} \nonumber 
		|N_{\gamma}^{j+1}| \lesssim \sum_{k=j+1}^{\infty}{(2^kr)^{|\gamma|+n \left( 1-\frac{1}{p} \right) } \, (2^k)^{-\frac{\lambda}{t}+n \left( \frac{1}{p}-\frac{1}{t} \right)}} \lesssim  |B_j|^{1-\frac{1}{p}} \, (2^jr)^{|\gamma|} \, (2^j)^{-\frac{\lambda}{t}+n\left( \frac{1}{p}-\frac{1}{t} \right)} .
	\end{align}
	In that way, using that $(2^jr)^{|\gamma|} |\phi_{\gamma}^j(x)|\leq C$ uniformly in $j$ and the previous control follows
	$$
	|N^{j}_{\gamma}|E_j|^{-1}\phi_{\gamma}^{j}(x)| \lesssim  |B_j|^{-\frac{1}{p}}(2^j)^{-\frac{\lambda}{t}+n \left( \frac{1}{p}-\frac{1}{t} \right)}.
	$$
	Denote $\psi_{\gamma}^{j}(x)=h_{j} \,\, B_{j\gamma}(x)$ where $h_{j}=(2^j)^{-\frac{\lambda}{t}+n \left( \frac{1}{p}-\frac{1}{t} \right)}$ and $B_{j\gamma}(x)=k_j \, \psi_{\gamma}^{j}(x)$ for $k_j = (2^{j})^{\frac{\lambda}{t}-n \left( \frac{1}{p}-\frac{1}{t} \right)}$. It is clear that $B_{j\gamma}$ is a multiple of a $(p,\infty)$-atom since $\supp(B_{j\gamma}) \subset B_j$, $\|B_{j\gamma}(x)\|_{L^\infty} \lesssim |B_j|^{-\frac{1}{p}}$ and the moment condition follows immediately from \eqref{AF-07}. Clearly
	\begin{equation*}
		\sum_{j=0}^{\infty}{|h_{j}|^p} = \sum_{j=0}^{\infty}{(2^j)^{-\frac{\lambda p}{t}+n \left( 1-\frac{p}{t} \right)}}< \infty.
	\end{equation*}

	The case $r > 1$ follows the same steeps with crucial control 
	\begin{align*}
		\| M_j \|_{L^t}  \lesssim  (2^j)^{-\frac{\lambda}{t}} \ r^{\frac{\lambda}{t}+n \left( \frac{1}{t}-\frac{1}{p}\right)} 
		\lesssim  |B_j|^{\frac{1}{t}-\frac{1}{p}} \ (2^j)^{-\frac{\lambda}{t}+n \left( \frac{1}{p}-\frac{1}{t} \right)}
	\end{align*}
	from condition (M2). We point out that the case $t=1$ follows by the same argument as before with the restriction $0<p<1$.
	
	Summarizing we have $\displaystyle M(x) = \sum_{j =0}^{\infty}{\gamma_j a_j(x)}$ converges in $L^{t}(\R^n)$, consequently in $\mathcal{S}'(\R^n)$, where $a_{j}$ are $(p,t)-$atoms and  $\{ \gamma_j \}_{j \in \N}$ is a sequence of complex scalars such that $\left(\sum_{j}{|\gamma_j|^p}\right)^{1/p}\leq C$, with constant independent of $M$. Clearly the series converges in $H^p(\R^n)$ 
	and $\| M \|_{H^p} \leq C$.
\end{proof}

\section{Proof of Theorems \ref{teorema_s-hormander} and \ref{theorem-Linf-BMO} }\label{secfour}

We starting with a notation that will be useful in the proof of Theorem \ref{teorema_s-hormander}.

\begin{definition} \label{parametros_admissiveis_s}
	Let $\lambda>0$, $1\leq s_1 < \infty$, $1 < s_2 < \infty$, $s_1 \leq s_2$, $0<p, \rho < 1$, $0<\sigma \leq 1$ and $\beta$ such that $(1-\sigma)\left( 1-\dfrac{1}{s_2} \right)n \leq \beta < \left( 1-\dfrac{1}{s_2} \right)n$. We say that $(p,\rho, \beta , \lambda,s_1,s_2)$ are admissible parameters when the relation 
	$$
	\displaystyle{n \left( \frac{s_1}{p}-1 \right) < \lambda \leq n \left( \frac{s_1}{s_2}-1\right) + \frac{s_1\beta}{(1-\rho)} }
	$$
	holds. For $p=1$ we restrict $1<s_1<\infty$.
\end{definition}

\begin{proof}[Proof of Theorem \ref{teorema_s-hormander}]
	Let $a(x)$ be a $(p,\infty)-$atom supported on $B(z,r) \subset \R^n$. From Lemma \ref{molecular_s_decomposition}, it is enough to show that $Ta$ is a $(p,\rho, q , \lambda,s_2,s_1)-$molecule where 
	$$
	\frac{1}{q}=\frac{1}{s_2}+\frac{\beta}{n} \quad \quad   \text{and} \quad   \quad (p,\rho, \beta , \lambda,s_1,s_2) \,\,\, \text{are admissible parameters}
	$$
	for some $\rho \leq \sigma$ to be chosen. We will restrict ourselves to the case $0<p\leq 1$ and $1<s_1 \leq 2$. The comments for $s_1=1$ will be done at the end. Suppose first $r > 1$. From Remark \ref{remark-molecules}-(b), we will show that (M1a) and (M2a) holds. Since $T$ is bounded on $L^2(\R^n)$ we have
	\begin{align*} \label{Ls_molecula}
		\int_{B(z,2r)}{|Ta(x)|^{s_1}dx} &\leq |B(z,2r)|^{1-\frac{s_1}{2}} \, \| Ta \|_{L^2}^{s_1} 
		\lesssim  \, |B(z,2r)|^{1-\frac{s_1}{2}} \, \| a \|_{L^2}^{s_1} 
		\lesssim  \ r^{n \left(1-\frac{s_1}{p}  \right)}.
	\end{align*}
	To estimate (M2a), the moment condition of the atom allow us to write
	\begin{align*}
		\int_{\R^n \setminus B(z,2r)}|Ta(x)|^{s_1}&|x-z|^{\lambda}dx \\
		&= \sum_{j=0}^{\infty}{\int_{C_j(z,r)}{\left| \int_{B(z,r)}{[K(x,y)-K(x,z)]a(y)dy} \right|^{s_1} \, |x-z|^{\lambda}dx}}.
	\end{align*}
	Then, from Minkowski inequality for integrals, H\"older's inequality and the assumption \eqref{s-hormander-1} we may control the previous sum as 
	\begin{align*}
		& \sum_{j=0}^{\infty}{\left\{ \left[ \int_{C_j(z,r)}{ \left( \int_{B(z,r)}{|K(x,y)-K(x,z)|\,\, |a(y)| \,\, |x-z|^{\frac{\lambda}{s_1}}dy} \right)^{s_1}  dx}   \right]^{\frac{1}{s_1}} \right\}^{s_1}} \nonumber \\
		& \leq  \sum_{j=0}^{\infty}{ \left\{ \int_{B(z,r)}{ \left[ \int_{C_j(z,r)}{|K(x,y)-K(x,z)|^{s_1}\,\, |a(y)|^{s_1} \,\, |x-z|^{\lambda}dx} \right]^{\frac{1}{s_{1}}}dy } \right\}^{s_1}} \nonumber \\
		& \leq  \sum_{j=0}^{\infty}{ (2^jr)^{\lambda} \left\{ \int_{B(z,r)}{|a(y)| \left[ \int_{C_j(z,r)}{|K(x,y)-K(x,z)|^{s_1} dx} \right]^{\frac{1}{s_1}} dy} \right\}^{s_1}} \nonumber \\
		& \leq  \sum_{j=0}^{\infty}{(2^jr)^{\lambda} \,\, (2^jr)^{-n(s_1-1)} \,\, 2^{-j s_1\delta}} \,\,  \left( \int_{B(z,r)}{|a(y)|dy} \right)^{s_1} \nonumber \\
		& \leq  \sum_{j=0}^{\infty}{(2^jr)^{\lambda} \,\, (2^jr)^{-n(s_1-1)} \,\, 2^{-j s_1\delta}} \,\, r^{s_1n\left( 1-\frac{1}{p} \right)} \nonumber \\
		& = C \ r^{\lambda + n\left( 1-\frac{s_1}{p} \right)} \sum_{j=0}^{\infty}{2^{j\left[\lambda-n(s_1-1)-s_1\delta \right]}} \nonumber \\
		& =  C \ r^{\lambda + n\left( 1-\frac{s_1}{p} \right)} \nonumber 
	\end{align*}
	assuming $\lambda < n(s_1-1)+s_1\delta$ and this concludes the proof of (M1a) and (M2a) for $1<s_1\leq 2$. In the case $2<s_1<\infty$, assuming the additional hypothesis (iii) for $T^{\ast}$, we obtain from Remark \ref{1.2.a} the continuity in $L^{s_1}(\R^n)$. From this, the condition (M1a) follows by
	\begin{equation}\nonumber
		\int_{B(z,2r)}{|Ta(x)|^{s_1}dx} \lesssim  \| a \|_{L^{s_1}}^{s_1}  \lesssim  \ r^{n \left(1-\frac{s_1}{p}  \right)}
	\end{equation}
	and (M2a) follows analogously.
	
	Now, we are moving on to the case $r \leq 1$. Since $s_1 \leq s_2$ and $T$ is bounded from $L^q(\R^n)$ to $L^{s_2}(\R^n)$ we have
	\begin{align*} \label{Lq-Ls-ineq}
		\int_{B(z,r^{\rho})}{|Ta(x)|^{s_1}dx} &\leq |B(z,r^{\rho})|^{1-\frac{s_1}{s_2}} \| Ta \|_{L^{s_2}}^{s_1} 
		\lesssim |B(z,r^{\rho})|^{1-\frac{s_1}{s_2}} \| a \|_{L^q}^{s_1} \\
		& \lesssim r^{n \left[\rho \left(1-\frac{s_1}{s_2}\right)+s_1 \left(\frac{1}{q}-\frac{1}{p} \right) \right]}
	\end{align*}
	which proves (M3a). For (M4a), using the same argument as before and \eqref{s-hormander-2} follows
	\begin{align*}
		\int_{\R^n \setminus B(z,r^{\rho})}&{|Ta(x)|^{s_1}|x-z|^{\lambda}dx} \\
		& \leq \sum_{j=0}^{\infty}{ (2^jr^{\rho})^{\lambda} \left\{ \int_{B(z,r)}{|a(y)| \left[ \int_{C_j(z,r^{\rho})}{|K(x,y)-K(x,z)|^{s_1} dx} \right]^{\frac{1}{s_1}} dy} \right\}^{s_1}} \\
		& \lesssim \ \sum_{j=0}^{\infty}{(2^jr^{\rho})^{\lambda} \,\, \left(|C_j(z,r^{\rho})|^{{\frac{1}{s_1}-1+\frac{\delta}{n} \left(\frac{1}{\rho}-\frac{1}{\sigma} \right)}} \,\, 2^{-\frac{j\delta}{\rho}} \right)^{s_1} } \,\, \| a \|_{L^\infty}^{s_1} \,\, |B(z,r)|^{s_1} \nonumber \\
		& \lesssim \ \sum_{j=0}^{\infty}{(2^jr^{\rho})^{\lambda} \,\, \left(|C_j(z,r^{\rho})|^{{\frac{1}{s_1}-1+\frac{\delta}{n} \left(\frac{1}{\rho}-\frac{1}{\sigma} \right)}} \,\, 2^{-\frac{j\delta}{\rho}} \right)^{s_1} \,\, r^{s_1n\left( 1-\frac{1}{p} \right) }} \nonumber \\
		& = C \  r^{\rho\lambda + n \left[ s_1+\frac{s_1\delta}{n}-s_1\rho \left(1-\frac{1}{s_1}+\frac{\delta}{n\sigma}\right) -\frac{s_1}{p} \right]} \sum_{j=0}^{\infty}{2^{j\left[ \lambda-n(s_1-1)-\frac{s_1\delta}{\sigma} \right]}} \nonumber \\
		& \lesssim \ r^{\,\rho\lambda + n \left[ \rho \left(1-\frac{s_1}{s_2} \right)+s_1 \left(\frac{1}{q}-\frac{1}{p}\right) \right]}, \nonumber
	\end{align*}
	where the convergence of the series is evident since $0<\sigma\leq 1$, $ \lambda <n(s_1-1)+s_1\delta$ assumed previously and $\rho$ is chosen such that 
	$$ 
	s_1+\frac{s_1\delta}{n}-\rho \left(s_1-1+\frac{s_1\delta}{n\sigma}\right)= \rho \left( 1-\frac{s_1}{s_2} \right)+ \frac{s_1}{q} \ \Longleftrightarrow \ \rho := \frac{n \left( 1-\frac{1}{q}\right)+\delta}{n \left( 1-\frac{1}{s_2} \right)+\frac{\delta}{\sigma}}.
	$$ 
	We point out that $\rho \leq \sigma$ since $\beta \geq n(1-\sigma)\left( 1-{1}/{s_2} \right)$ is equivalent to ${1}/{q} \geq 1+\sigma \left({1}/{s_2}-1 \right)$ and then
	\begin{equation} \label{choice-rho}
	n \left( 1-\frac{1}{q} \right) + \delta \leq n \sigma \left( 1- \frac{1}{s_2} \right) + \delta \Longleftrightarrow \rho = \frac{n \left( 1-\frac{1}{q}\right)+\delta}{n \left( 1-\frac{1}{s_2} \right)+\frac{\delta}{\sigma}} \leq \sigma.
	\end{equation}
	Summing up, from the admissible parameters it follows that $n\left( {s_1}/{p}-1 \right)<\lambda \leq n \left( {s_1}/{s_2}-1\right) + {s_1\beta}/{(1-\rho)}$ and  $\lambda<n(s_1-1)+s_1\delta$. Note that ${s_2\beta}/{(1-\rho)}<(s_2-1)n+s_2\delta$ since
	\begin{align*}
		\beta< \frac{\delta \left( \frac{1}{\sigma}-1 \right) \left[ n\left( 1-\frac{1}{s_2} \right)+\delta  \right]}{\delta \left(\frac{1}{\sigma}-1\right)} 
		&\Longleftrightarrow  \frac{s_2\beta \left[n\left( 1-\frac{1}{s_2} \right)+\frac{\delta}{\sigma} \right]}{\beta+\frac{\delta}{\sigma}-\delta}< (s_2-1)n+s_2\delta \nonumber \\
	\end{align*}
	and this implies that
	\begin{align} \label{control-lambda}
		n \left( \frac{s_1}{s_2}-1\right) + \frac{s_1\beta}{(1-\rho)} < n(s_1-1)+s_1\delta.
	\end{align}
	This relation naturally implies a lower bound of $p$ given by 	
	\begin{equation*}
		n \left(\frac{s_1}{p}-1 \right)<n \left(\frac{s_1}{s_2}-1 \right)+ \frac{s_1\beta}{1-\rho} \Longleftrightarrow \frac{1}{p}< \frac{1}{s_2}+\frac{\beta \left[ \frac{\delta}{\sigma}+n\left(1-\frac{1}{s_2}\right)\right]}{n\left(\frac{\delta}{\sigma}-\delta+\beta\right)}:=\frac{1}{p_{_0}}.
	\end{equation*}
	The argument for $s_1=1$ and $0<p<1$ follows in the same way with minor changes.
	
	Clearly $n/(n+\delta)<p_{_0}$ thus for $p_{_0}<p\leq 1$ follows $n(1/p-1)<\delta$. Since $T^{\ast}(x^{\alpha})=0$ for $|\alpha| \leq \lfloor \delta \rfloor$, then (M5) is trivially valid.
\end{proof}

\begin{proof}[Proof of Theorem \ref{theorem-Linf-BMO}]
	Let $f \in L^{\infty}(\R^n)$. In order to prove that $T$ maps continuously $L^{\infty}(\R^n)$ into $BMO(\R^n)$ we need to show that for any ball $B \subset \R^n$ there exist a constant $a_B$ (may depend on $B$) such that
	$$
	\sup_{B}{ \frac{1}{|B|} \int_{B}|Tf(x)-a_B|dx} \leq C \, \| f \|_{L^{\infty}}
	$$
	where $C>0$ is a constant independent of $B$ (the left-hand side of the previous inequality define a norm in $BMO(\R^n)$). To do so, let $B:=B(z,r) \subset \R^n$ and suppose $r\leq 1$. Split $f$ into
	$$
	f = f \chi_{_{B(z,2r^{\sigma})}} + f \chi_{_{\R^n \setminus B(z,2r^{\sigma})}} := f_1+f_2.
	$$
	Since $T^{\ast}: L^q(\R^n) \rightarrow L^{s_2}(\R^n)$ is bounded, then $T: L^{s_2'}(\R^n) \rightarrow L^{q'}(\R^n)$ will also be bounded for $\dfrac{1}{s_2'}=1-\dfrac{1}{s_2}$ and $\dfrac{1}{q'}=\dfrac{1}{s_2'}-\dfrac{\beta}{n}$. In particular, since $f_1 \in L^{s_2'}(\R^n)$ then $Tf_1$ is well defined, belongs to $L^{q'}(\R^n)$ and 
	\begin{equation}
		\int_{B(z,r)}{|Tf_1(x)|dx} \leq  |B(z,r)|^{\frac{1}{q}}   \| Tf_1 \|_{L^{q'}} 
		\lesssim   |B(z,r)|^{\frac{1}{q}}  \| f_1 \|_{L^{s_{2'}}}
		\lesssim    |B(z,r)|^{\frac{1}{q}+\frac{\sigma}{s_{2'}}}  \| f \|_{L^\infty}.
	\end{equation}
	Since ${1}/{q}+{\sigma}/{s_2'}-1\geq 0$ and $r\leq 1$ we have $|B(z,r)|^{\frac{1}{q}+\frac{\sigma}{s_2'}} \lesssim1$. Then
	\begin{equation} \label{Linf_BMO1}
		\frac{1}{|B(z,r)|} \int_{B(z,r)}{|Tf_1(x)|dx} \leq C \ \| f \|_{L^\infty}.
	\end{equation}
	For $f_2$ we use condition \eqref{s-hormander-2a} to show 
	\begin{align*}
		\int_{\R^n}{|K(x,y)-K(z,y)| \,\, |f_2(y)|dy} & \leq \ \| f \|_{L^\infty} \int_{|y-z|>2r^{\sigma}}{|K(x,y)-K(z,y)|dy} \\
		& =  \ \| f \|_{L^\infty} \sum_{j=0}^{\infty}{\int_{C_j(z,r^{\sigma})}{|K(x,y)-K(z,y)|dy}} \\
		&\leq \ \| f \|_{L^\infty} \sum_{j=0}^{\infty}{2^{-\frac{j\delta}{\sigma}}} 
		\lesssim     \| f \|_{L^\infty}
	\end{align*}
	and from previous estimate
	\begin{equation} \label{Linf_BMO2}
		\frac{1}{|B(z,r)|} \int_{B(z,r)}{|Tf_2(x)-Tf_2(z)|dx} \leq C \ \| f \|_{L^\infty}.
	\end{equation}
	Hence, we choose $a_Q := Tf_2(z)$ and from \eqref{Linf_BMO1} and \eqref{Linf_BMO2} we conclude
	\begin{align*}
		\frac{1}{|B(z,r)|} \int_{B(z,r)}{|Tf(x)-Tf_2(z)|dx} &\leq \frac{1}{|B(z,r)|} \int_{B(z,r)}{|Tf_1(x)|+|Tf_2(x)-Tf_2(z)|dx} \\
		&\leq C \, \| f \|_{L^{\infty}}.
	\end{align*}
	The proof for $r>1$ is analogous if we split $f$ in $B(z,2r)$ and $\R^n \setminus B(z,2r)$. We point out that only \eqref{s-hormander-1a} and $L^2-$boundedness of $T$ is required for this case.
\end{proof}

\subsection{Comments and remarks} \label{comment}

The conclusion of Theorem \ref{teorema_s-hormander} is still open for $p=p_{_0}$, however under assumption $s=2$ and $D_{1}$ condition, the \cite[Theorem 3.9]{AlvarezMilmanVectorValued} asserts that  $T$ can be extended to a bounded operator from $H^{p}(\R^n)$ to $L^{p}(\R^n)$ for $p_{_{0}} \leq p\leq 1$. An extension of this result is presented at Corollary \ref{coro}.

We point out the condition  $T^{\ast}(x^{\alpha})=0$ for $|\alpha| \leq \lfloor \delta \rfloor$ assumed at Theorem  \ref{teorema_s-hormander}  can be refined to $|\alpha| \leq N_{p_{_0}}:= \lfloor n(1/p_{_0}-1) \rfloor$.  Furthermore, this assumption is a necessary condition at Theorem \ref{teorema_s-hormander}. Suppose that $T$ maps continuously $H^p(\R^n)$ to itself for all $p_{_0}<p\leq1$ and let $f \in L^{2}_{\#, N_{p_{_0}}}(\R^n)$. Since $f$ is a multiple of a $(p,2)-$atom follows $Tf$ satisfies the conditions (M1)-(M4) as a particular case of Proposition \ref{molecule-local-integrable}, thus $Tf \in L^1(\R^n) \cap H^p(\R^n)$. Hence, by \cite[Sec. 5.4 (c) p.128 ]{SteinHarmonic} it follows that
$$
\int{Tf(x)x^{\alpha}dx} = 0, \ \text{for all} \ |\alpha|\leq N_p \ \text{and} \ p_{_0}<p\leq1.
$$
Therefore it will also holds for $N_{p_{_0}}$ since $p \searrow p_{_0}$.

In Theorem \ref{theorem-Linf-BMO}, the crucial hypothesis that $T^{\ast}$ is a bounded operator from $L^q(\R^n)$ to $L^{s_2}(\R^n)$ may be weakened by the condition
$$
|B(z,r)|^{-\frac{1}{q}} \int_{B(z,r)}{|Tf(x)|dx} \leq C \, \| f \|_{L^{{s_2}' }}
$$
i.e. $T$ maps continuously $L^{{s_2}' }(\R^n)$ into $\mathcal{M}^{1}_{\lambda}(\R^n)$ where ${1}/{\lambda}={1}/{s_2'}-{\beta}/{n}$ and 
$$
\mathcal{M}^{1}_{\lambda}(\R^n)=\left\{ f \in L^{1}_{loc}(\R^n): \, \sup_{0<r \leq 1} |B(z,r)|^{\frac{1}{\lambda}-1} \int_{B(z,r)}{|f(x)|dx}< \infty \right\}
$$
denotes the local Morrey-space with $\lambda>1$.

\subsection{Weaker integral derivative conditions} \label{section-derivative-conditions}

In this section, we consider \textit{strongly singular Calder\'on-Zygmund operators of type $\sigma$} associated  kernels satisfying derivative conditions. Let $\delta >0 $ and $K \in C^{\lfloor \delta \rfloor}$ away the diagonal on $\R^{2n}$ satisfying 

\begin{equation} \label{derivative-condition}
	|\partial_{y}^{\gamma}K(x,y)-\partial_{y}^{\gamma}K(x,z)| + |\partial_{y}^{\gamma}K(y,x)-\partial_{y}^{\gamma}K(z,x)| \leq C \frac{|y-z|^{\delta-\lfloor \delta \rfloor}}{|x-z|^{n+\frac{\delta}{\sigma}}}, 
\end{equation}
for $\gamma \in \Z^{n}_{+}$ with  $|\gamma|=\lfloor \delta \rfloor$, $|x-z| \geq 2|y-z|^{\sigma}$ and $0<\sigma \leq 1$.

Condition \eqref{derivative-condition} is a natural generalization of derivative conditions usually assumed on standard $\delta-$kernels of type $\sigma$  
satisfying \eqref{pontual_kernel_tipo_sigma} (see \cite[p. 320]{GarciaFranciaWeighted} and \cite[p. 117]{SteinHarmonic}). In the same way, we may replace the weaker integral $D_{s}$ condition 
\eqref{s-hormander-1} and \eqref{s-hormander-2} by the \textit{derivative $D_{s}$ condition} 

\begin{align}\label{hh1}
	&\left( \int_{C_j(z,r)}{|\partial_{y}^{\gamma}K(x,y)-\partial_{y}^{\gamma}K(x,z)|^s+|\partial_{y}^{\gamma}K(y,x)-\partial_{y}^{\gamma}K(z,x)|^sdx} \right)^{\frac{1}{s}} \nonumber  \\  
	& \quad \quad \quad \quad \quad \quad \quad \quad \quad \quad \quad \quad \quad \quad \quad \quad \lesssim \, r^{-\lfloor \delta \rfloor} \, |C_j(z,r)|^{\frac{1}{s}-1} \, 2^{-j\delta} 
\end{align}
if $r>1$ and
\begin{align}\label{hh2}
	&\left( \int_{C_j(z,r^{\rho})}{|\partial_{y}^{\gamma}K(x,y)-\partial_{y}^{\gamma}K(x,z)|^s+|\partial_{y}^{\gamma}K(y,x)-\partial_{y}^{\gamma}K(z,x)|^sdx} \right)^{\frac{1}{s}} \nonumber  \\  
	& \quad \quad \quad \quad \quad \quad \quad \quad \quad \quad \quad \quad \quad \lesssim \, r^{-\lfloor \delta \rfloor} \, |C_j(z,r^{\rho})|^{\frac{1}{s}-1+\frac{\delta}{n} \left(\frac{1}{\rho}-\frac{1}{\sigma}  \right)}  2^{-\frac{j\delta}{ \rho}}
\end{align}
if $r<1$.

We announce the following self-improvement of Theorem \ref{teorema_s-hormander}:

\begin{theorem} \label{theorem-derivative-condition}
	Let $T: \mathcal{S}(\R^{n}) \rightarrow \mathcal{S}'(\R^{n})$ be a bounded linear operator and suppose that $T$ satisfies assumptions (i) and (iii) from Theorem \ref{teorema_s-hormander} and
	\begin{enumerate}
		\item[\textnormal{(ii)'}] $T$ is associated to a kernel satisfying the derivative $D_{s_{1}}$ condition \eqref{hh1} and \eqref{hh2}; 
	\end{enumerate}
	Then, if $T^{\ast}(x^{\alpha})=0$ for all $|\alpha| \leq \lfloor \delta \rfloor $, $1< s_1 \leq 2$ and $s_1 \leq s_2$, the operator $T$ is bounded from $H^p(\R^n)$ to itself for $p_{_0}<p \leq 1$, where $p_{0}$ is given \eqref{pcritico}. Moreover, if $T^{*}$ also satisfies (iii) then the conclusion holds for $1<s_{1}<\infty$ and $s_{1} \leq s_{2}$. The case $s_{1}=1$ also holds, however only for $p_{_{0}}<p<1$.
\end{theorem}

\noindent The proof of the previous result is analogous of Theorem \ref{teorema_s-hormander} since Taylor's formula allows us to write $Ta(x) = \int_{B(z,r)} R(x,y)a(y)dy$, in which 
$$
R(x,y) = \sum_{|\gamma|=M} \frac{(y-z)^{\gamma}}{\gamma !} \left[ \partial_{y}^{\gamma}K(x,\xi_{y})-\partial_{y}^{\gamma}K(x,z) \right]
$$
for some $\xi_{y}$ in the line segment between $y$ and $z$.

Examples of operators satisfying such kernel conditions will be discussed in Section \ref{subsection:pseudo}.
%
%

\section{Applications} \label{section_aplication}

\subsection{Weighted continuity} \label{subsection:weights}

In this section we show that, under the hypothesis of Theorem \ref{teorema_s-hormander}, strongly singular Calder\'on-Zygmund operators are bounded from $H^{p}_{w}(\R^n)$ to $L^{p}_{w}(\R^n)$, where $w$ belongs to a special class of Muckenhoupt weight. 

A non-negative measurable function $w(x)$ belongs to the class $A_1$ if {there exists $C>0$} such that for any ball $B \subset \R^n$ we have 
\begin{equation}
	\frac{1}{|B|} \int_{B}{w(y)dy} \leq C \, w(x), \quad \text{for a.e. } x \in B.
\end{equation}
In comparison with Lebesgue measure, if $w \in A_{1}$ then  there exists $c>0$ such that $ |E|w(B)\leq c |B|w(E)$ for any $E \subseteq B$ in which $B \subset \R^n$ and $w(B):=\int_{B}w(x)dx$ (see \cite[Chapter IV.2 - Theorem 2.1 (b)]{GarciaFranciaWeighted}). We say the weight $w(x)$ satisfies the reverse H\"older inequality for $1<r<\infty$, simply denoted by $w \in RH_{r}$, if there exists a constant $C>0$ such that 
\begin{equation*}
	\left( \frac{1}{|B|} \int_{B}w^{r}(x)dx \right)^{\frac{1}{r}} \leq \frac{C}{|B|} \int_{B} w(x)dx
\end{equation*}
for any ball $B\subset \R^n$. Clearly if $w \in RH_{r}$ then $w \in RH_{s}$ for all $1<s<r$. 


We denoted the weighted Lebesgue space $L^{p}_{w}(\R^n):=L^{p}(\R^n,w(x)dx)$ the set of all measurable functions such that
$$
\| f \|_{L^{p}_{w}} := \left( \int_{\R^n}{|f(x)|^p w(x)dx} \right)^{\frac{1}{p}} < \infty.
$$
In the same spirit, we define the weighted Hardy space, denoted by
$H^{p}_{w}(\R^n)$, to be the set of tempered distributions $f \in \S'(\R^n)$ such that  $\mathcal{M}_{\varphi}f \in L^{p}_{w}(\R^n)$ {and $\| f \|_{H^{p}_{w}} := \| M_{\varphi}f \|_{L^{p}_{w}}$ denotes its quasi-norm.} 
We refer \cite{strombergWeighted1989} for further details on the  weighted Hardy space.

\begin{definition}\cite[p. 112]{strombergWeighted1989}
	Let $0<p\leq1$ and $w \in A_1$. We say that a measurable function $a(x)$ is a $(w,p,\infty)-$ atom if there exists $B(z,r) \subset \R^n$ such that
	$$
	\supp(a)\subset B(z,r), \ \ \| a \|_{L^{\infty}} \leq w(B(z,r))^{-\frac{1}{p}} \ \ \text{and} \ \ \int{a(x)x^{\alpha}dx}=0 
	$$
	for any multi-index such that $|\alpha|\leq N_p$.
\end{definition}

{\noindent If $w \in A_1$ and $f \in H^{p}_{w}(\R^n)$, then there exist a sequence of coefficients $\{ \lambda_j\}_{j}$ and $(w,p,\infty)-$atoms $\{a_j\}_{j}$ such that $f=\sum \lambda_j a_j$, where the convergence is in $H^{p}_{w}-$norm. Moreover, $
	\inf \left\{ \left( \sum_{j \in \N}{|\lambda_j|^{p}} \right)^{{1}/{p}} \right\} \approx \| f \|_{H^{p}_{w}}
	$
	where the infimum is taken over all such atomic representations of $f$. In addiction, the converse of this result is also true. For a more general case of this result see \cite[Chapter VIII, Theorem 1]{strombergWeighted1989}}.

{Now we present the proof of Theorem \ref{th5.2}}


%



\begin{proof}[Proof of Theorem \ref{th5.2}]
	Let $a(x)$ be a $(p, \infty)-$atom in $H^{p}_{w}(\R^n)$ supported on $B(z,r)$. We will show that $Ta$ is uniformly bounded in $L^{p}_{w}-$norm. Suppose first $r>1$ and split
	\begin{align*}
		\| Ta \|_{L^{p}_{w}}^{p} = \int_{B(z,2r)}{|Ta(x)|^p w(x)dx}+ \sum_{j=1}^{\infty}{\int_{C_j(z,r)}{|Ta(x)|^p w(x)dx}}:=\textnormal{I}_1+\textnormal{I}_2.
	\end{align*}
	The first integral can be uniformly estimated from H\"older inequality with exponent $2/p$, the $L^2-$continuity of $T$ and from $w\in RH_{2/(2-p)}$. In fact,
	\begin{align*}
		\textnormal{I}_1 &\leq \left( \int{|Ta(x)|^2dx} \right)^{\frac{p}{2}} \left( \frac{1}{|B(z,2r)|} \int_{B(z,2r)}{w^{\frac{2}{2-p}}(x)dx} \right)^{1-\frac{p}{2}} \, |B(z,2r)|^{1-\frac{p}{2}} \\
		&\lesssim \frac{w(B(z,2r))}{w(B(z,r))} |B(z,r)|^{\frac{p}{2}} |B(z,2r)|^{-\frac{p}{2}} \\
		&\lesssim \frac{|B(z,2r)|}{|B(z,r)|} |B(z,r)|^{\frac{p}{2}} |B(z,2r)|^{-\frac{p}{2}} \lesssim 1.
	\end{align*}
	For the second integral note first that since $w\in RH_{s_1/p(s_1-1)}$ it follows
	\begin{align}
		\int_{C_j(z,r)} &|K(x,y)-K(x,z)|w^{\frac{1}{p}}(x)dx \leq \left( \int_{C_j(z,r)} |K(x,y)-K(x,z)|^{s_1}dx \right)^{\frac{1}{s_1}}  \times \nonumber \\
		&\quad \quad \quad \quad \times \left( \int_{C_j(z,r)}w^{\frac{s_1}{p(s_1-1)}}(x)dx \right)^{1-\frac{1}{s_1}} \nonumber \\
		&\lesssim 2^{-j\delta} \left(\frac{1}{|C_j(z,r)|} \int_{C_j(z,r)}w^{\frac{s_1}{p(s_1-1)}}(x)dx \right)^{1-\frac{1}{s_1}} \nonumber \\
		&\lesssim 2^{-j\delta} \left(\frac{|B_{j+1}(z,r)|}{|C_j(z,r)|} \right)^{1-\frac{1}{s_{1}}} |B_{j+1}(z,r)|^{-\frac{1}{p}} w(B_{j+1}(z,r))^{\frac{1}{p}}. \label{estimativaHp-Lp} 
	\end{align}

	Then,
	\begin{align*}
		\textnormal{I}_2 &\leq \sum_{j=1}^{\infty} \int_{C_j(z,r)} \left( \int_{B(z,r)}|K(x,y)-K(x,z)|\,|a(y)|dy \right)^{p}w(x)dx \\
		& \leq \sum_{j=1}^{\infty} w(B(z,r))^{-1} \left(\int_{C_j(z,r)} \int_{B(z,r)}|K(x,y)-K(x,z)|\,w^{\frac{1}{p}}(x)dydx \right)^{p} { |C_j(z,r)|^{1-p} }\\
		&= \sum_{j=1}^{\infty} w(B(z,r))^{-1} \left( \int_{B(z,r)} \left[ \int_{C_j(z,r)}|K(x,y)-K(x,z)|\,w^{\frac{1}{p}}(x)dx \right] dy \right)^{p} |C_j(z,r)|^{1-p} \\
		& \lesssim \sum_{j=1}^{\infty} \frac{w(B_{j+1}(z,r))}{w(B(z,r))}  |C_j(z,r)|^{1-p} |B(z,r)|^{p} \left(\frac{|B_{j+1}(z,r)|}{|C_j(z,r)|} \right)^{p \left(1-\frac{1}{s_{1}}\right)} \times \\
		& \quad \quad  \quad \quad \quad \quad \quad \quad \quad \quad \quad \quad \quad \quad \quad \quad \quad \quad \quad \quad \quad \quad   \times |B_{j+1}(z,r)|^{-1} \, 2^{-jp\delta} \\
		&\lesssim \sum_{j=1}^{\infty} |C_j(z,r)|^{1-p-p\left(1-\frac{1}{s_1}\right)} |B(z,r)|^{p-1} |B_{j+1}(z,r)|^{p \left(1-\frac{1}{s_{1}}\right)} 2^{-jp\delta} \\
		&\lesssim \sum_{j=1}^{\infty} 2^{j[n-p(n+\delta)]} \lesssim 1
	\end{align*}
	since $p >n/(n+\delta)$. Lets consider now the case $0< r \leq 1$. In the same way, we split
	\begin{align*}
		\| Ta \|_{L^{p}_{w}}^{p} = \int_{B(z,2r^{\rho})}{|Ta(x)|^p w(x)dx}+ \sum_{j=1}^{\infty}{\int_{C_j(z,r^{\rho})}{|Ta(x)|^p w(x)dx}}:=\textnormal{I}_3+\textnormal{I}_4
	\end{align*}
	for some $0< \rho \leq \sigma$ that will be chosen conveniently later. For the first integral, using H\"older inequality with exponent $s_2/p$, the $L^q-L^{s_2}$ continuity of $T$ and $w \in RH_{s_2/(s_2-1)}$ implies
	\begin{align*}
		\int_{B(z,2r^{\rho})}|Ta(x)|^p& w(x)dx \\
		&\leq \| Ta \|_{L^{s_2}}^{p} \left( \frac{1}{|B(z,2r^{\rho})|} \int_{B(z,2r^{\rho})}{w^{\frac{s_2}{s_{2}-p}}(x)dx} \right)^{1-\frac{p}{s_{2}}} |B(z,2r^{\rho})|^{1-\frac{p}{s_{2}}} \\
		&\lesssim \frac{w(B(z,2r^{\rho}))}{w(B(z,r))} \, |B(z,r)|^{\frac{p}{q}} |B(z,2r^{\rho})|^{-\frac{p}{s_{2}}} \\
		&\lesssim |B(z,r)|^{\frac{p}{q}-1} \,  |B(z,2r^{\rho})|^{1-\frac{p}{s_{2}}} \\
		&\lesssim r^{n\left[ \frac{p}{q}-1+\rho\left(1-\frac{p}{s_{2}}\right)  \right] } \lesssim 1 
	\end{align*}
	for
	$
	\rho \geq \rho_1:= \frac{1-p\left(\frac{1}{s_{2}}+\frac{\beta}{n} \right)}{1-\frac{p}{s_{2}}}.
	$
	For the second integral, proceeding just like in \eqref{estimativaHp-Lp}, it follows from $w\in RH_{s_1/p(s_1-1)}$ and $D_{s_1}$ condition that
	\begin{align*}
		\int_{C_j(z,r^{\rho})} &|K(x,y)-K(x,z)|w^{\frac{1}{p}}(x)dx \\
		&\lesssim |C_j(z,r^{\rho})|^{\frac{1}{s_1}-1+\frac{\delta}{n} \left( \frac{1}{\rho}-\frac{1}{\sigma} \right)}
		|B_{j+1}(z,r^{\rho})|^{1-\frac{1}{s_{1}}-\frac{1}{p}} w(B_{j+1}(z,r^{\rho}))^{\frac{1}{p}} \, 2^{-j\frac{\delta}{\rho}}. 
	\end{align*}
	Then,
	\begin{align*}
		\textnormal{I}_4 &\lesssim \sum_{j=1}^{\infty} \frac{w(B_{j+1}(z,r^{\rho}))}{w(B(z,r))} \, |B(z,r)|^{p} \, |C_j(z,r^{\rho})|^{\frac{p}{s_1}+1-2p+\frac{p\delta}{n}\left(\frac{1}{\rho}-\frac{1}{\sigma} \right)} \times \\
		& \quad \quad \quad \quad \quad \quad \quad \quad \quad \quad \quad \quad \quad \quad \quad  \times |B_{j+1}(z,r^{\rho})|^{p-\frac{p}{s_1}-1} 2^{-j\frac{p\delta}{\rho}} \\
		&\lesssim \sum_{j=0}^{\infty} \, |B(z,r)|^{p-1} \,|C_j(z,r^{\rho})|^{1-p+\frac{p\delta}{n}\left(\frac{1}{\rho}-\frac{1}{\sigma} \right)}	\, \left( \frac{|B_{j+1}(z,r^{\rho})|}{|C_j(z,r^{\rho})|} \right)^{p\left( 1-\frac{1}{s_1} \right)} \, 2^{-j\frac{p\delta}{\rho}} \\	
		&\lesssim \, r^{-\rho \left[ n(p-1)+\frac{p\delta}{\sigma} \right]+p\delta+n(p-1)} \, \sum_{j=1}^{\infty}{\, 2^{j \left[ n-p\left( n+\frac{\delta}{\sigma} \right) \right]}} \lesssim 1
	\end{align*}
	in which $\rho \leq \rho_2:= \frac{p(n+\delta)-n}{p \left( n+\frac{\delta}{\sigma} \right)-n} \leq \sigma$. The restriction $\rho_1 \leq \rho_2$ implies that uniform estimate holds for every $p \geq p_{_{0}}$. Hence, given $f \in H^{p}_{w}(\R^n)$, by standard arguments one has
	$$
	\|Tf \|_{L^{p}_{w}}^{p} \leq \sum_{j \in \N}|\lambda_j|^p \| Ta \|_{L^{p}_{w}}^{p} \lesssim \| f \|_{H^{p}_{w}}^{p},
	$$
	which concludes the proof.
\end{proof}

We remark that since $p_{_0}\leq p \leq 1$, we may replace the assumption  $w \in A_1 \cap RH_{d}$ for $d=\max\left\{\frac{s}{s-p}, \, \frac{s_1}{p(s_1-1)}\right\}$ at Theorem  \ref{th5.2} by the stronger condition $w \in A_1 \cap RH_{d_{_0}}$ for $d_{_0}=\max\left\{\frac{s}{s-1}, \, \frac{s_1}{p_{_0}(s_1-1)}\right\}$, where $s=\min \left\{2,s_{2} \right\}$.

A direct consequence of Theorem \ref{th5.2} is the following:

\begin{corollary}\label{coro}
	Let $T: \mathcal{S}(\R^{n}) \rightarrow \mathcal{S}'(\R^{n})$ be a linear and bounded operator as in Theorem \ref{teorema_s-hormander}. Then, $T$ can be extended to a bounded operator from $H^{p}(\R^n)$ to $L^{p}(\R^n)$  for $p_{_0} \leq p\leq 1$, with $p_{_0}$ given by \eqref{pcritico}.
\end{corollary}

\subsection{A special class of pseudodifferential operators} \label{subsection:pseudo}

It is well understood that pseudodifferential operators in the class $OpS^{m}_{\sigma,b}(\R^n)$ for certain parameters $m, \sigma$ and $b$ have distributional kernels satisfying the pointwise estimate \eqref{pontual_kernel_tipo_sigma} for $\delta=1$. This can be easily seen for operators $OpS^{-n(1-\sigma)}_{\sigma, b}(\R^n)$ from derivative estimate of the kernel presented in \cite[Theorem 1.1 (d)]{AlvarezHounie}. However, integral estimates are more suitable in dealing with this type of operators and using them we have the advantage of finding a wider set of examples. In \cite[Section 3]{AlvarezMilman}, the authors have shown that $OpS^{-m}_{\sigma,b}(\R^n)$ for $0<b\leq\sigma<1$ and $n(1-\sigma)/2\leq m < n/2$ satisfies a type of H\"ormander condition instead the pointwise condition \eqref{pontual_kernel_tipo_sigma}, i.e.,
\begin{equation}\nonumber
	\int_{|x|\geq 2r^{\sigma}}|K(x+z,x-y)-K(x+z,x)|dx+\int_{|x|\geq 2r^{\sigma}}|K(x-y,x+z)-K(x,x+z)|dx\leq C
\end{equation}
for all $z \in \R^n$, $|y|\leq r$ and $r\geq 0$. This represents the weakest condition known so far, but unfortunately it is still an open question to prove continuity on $H^p(\R^n)$ {from it} (see for instance the counterexample in \cite{YangYanDeng-Hormandercondition}).

In this section, we present classes of pseudodifferential operators satisfying the hypothesis of Theorems \ref{teorema_s-hormander} and \ref{theorem-derivative-condition}. We start verifying the derivative $D_{s_1}$ condition for $1 \leq s_1 \leq 2$, extending the case $s_{1}=1$ and $|\gamma|=0$ proved in \cite[Theorem 2.1]{AlvarezHounie}.


\begin{proposition} \label{exemplo_pseudo_2}
	Let $\delta>0$ and $T \in OpS^{m}_{\sigma, b}(\R^n)$ with $0<\sigma \leq 1$, $0\leq b <1$, $b \leq \sigma$ and $m \leq -n(1-\sigma)/2$. If $1 \leq s_1 \leq 2$, then $T$ satisfies the derivative $D_{s_1}$ condition with decay $\lfloor \delta \rfloor +1$. In particular, when $0<\delta < 1$ it satisfies integral conditions \eqref{s-hormander-1} and \eqref{s-hormander-2} with decay $1$.
\end{proposition}

Follows from  \cite[Theorem 3.5]{AlvarezHounie} that $T$ maps continuously $L^q(\R^n)$ into $L^{s_2}(\R^n)$ where $\dfrac{1}{q}=\dfrac{1}{s_2}+\dfrac{\beta}{n}$ and  $n(1-\sigma) \left( 1- \dfrac{1}{s_2} \right) \leq \beta < n\left( 1- \dfrac{1}{s_2} \right)$ since: 
\begin{enumerate}
	\item[($a_1$)] $\displaystyle	m \leq -\beta-n(1-\sigma)\left(\frac{1}{s_2}-\frac{1}{2}\right)$, if $1<q\leq s_2 \leq 2$;
	\item[($a_2$)] $m \leq -\beta$, if $1<q\leq 2 \leq s_2$;
	\item[($a_3$)] $m \leq -n(1-\sigma)/2$, if $2\leq q \leq s_2$.
\end{enumerate}
Note that $m \leq -n(1-\sigma)/2$ in all the cases and since $0 \leq b \leq \sigma<1$ we have that $T \in OpS^{m}_{\sigma,b}(\R^n)$ is bounded from $L^{2}(\R^n)$ to itself.  Now we present the proof of Proposition \ref{exemplo_pseudo_2}.

\begin{proof}
	Let $T \in OpS^{m}_{\sigma, b}(\R^n)$, $K$ its distributional kernel and we denote by $\widetilde{K}(x,y)= \partial_{y}^{\gamma}K(x,y)$ for $|\gamma|=\lfloor \delta \rfloor$. In order to obtain the derivative $D_{s_1}$ condition for $1 \leq s_1< 2$ is suffices to prove it for $s_1=2$. We claim that under the restriction $m \leq -n[(1-\sigma)/2+\lambda]$ in which $\lambda= \max \{0, (b-\sigma)/2  \}$ it follows for $r \geq 1$
	\begin{align*}
		\sup_{|y-z| \leq r}{\left( \int_{C_j(z,r^{})}{|\widetilde{K}(x,y)-\widetilde{K}(x,z)|^2dx} \right)^{\frac{1}{2}} \lesssim \, r^{-\lfloor \delta \rfloor} \, |C_j(z,r)|^{-\frac{1}{2}} \,\, 2^{-j(\lfloor \delta \rfloor + 1)}},
	\end{align*}
	and for $r<1$
	\begin{align*}
		&\sup_{|y-z| \leq r}\left( \int_{C_j(z,r^{\rho})}{|\widetilde{K}(x,y)-\widetilde{K}(x,z)|^2dx} \right)^{\frac{1}{2}} \\
		& \quad \quad \quad \quad \quad \quad \quad \quad \quad \quad   \lesssim \, r^{-\lfloor \delta \rfloor} \, |C_j(z,r^{\rho})|^{-\frac{1}{2}+ \frac{\lfloor \delta \rfloor + 1}{n} \left( \frac{1}{\rho}-\frac{1}{\sigma} \right)} \,\, 2^{-\frac{j}{\rho}(\lfloor \delta \rfloor+1)}
	\end{align*}
	
	The same estimate for the adjoint $\widetilde{K}(y,x)$ will be treated in the end assuming $m \leq -n(1-\sigma)/2 $.

	The proof consists an adaptation of \cite[Theorem 2.1]{AlvarezHounie}. Assume without loss of generality that the symbol $p(x,\xi)$ associated to the operator $T$ vanishes for $|\xi|\leq 1$ and consider $\psi \in C^{\infty}_{c}(\R)$ a non-negative function such that $\supp(\psi) \subset [1/2,1]$ and
	\begin{equation} \label{psi-1}
		\int_{0}^{\infty}{\psi \left( \frac{1}{t} \right) \,\, \frac{1}{t}\ dt } = \int_{1}^{2}{\psi \left( \frac{1}{t} \right) \,\, \frac{1}{t}\ dt }=1.
	\end{equation}
	Define $\displaystyle{K(x,y,t) = (2\pi)^{-n} \int{e^{i(x-y) \,\, \xi}p(x,\xi)\psi \left( \frac{|\xi|}{t} \right)d\xi}}$ and consequently 
	$$
	\widetilde{K}(x,y,t) = (-i)^{\lfloor \delta \rfloor}(2\pi)^{-n} \int{e^{i(x-y) \,\, \xi}  p(x,\xi)\psi \left( \frac{|\xi|}{t}  \right)\xi^{\gamma}d\xi}.
	$$
By the standard representation of the kernel of a pseudodifferential operator 
	and from \eqref{psi-1} we may write   
	\begin{equation} \label{q-hormander_demo3}
		\widetilde{K}(x,y)= \int_{0}^{\infty}{\widetilde{K}(x,y,t) \ \frac{dt}{t}} \ = \int_{1}^{\infty}{\widetilde{K}(x,y,t) \ \frac{dt}{t}}.
	\end{equation}
	Consider first $0<r<1$. From Minkowski inequality for integrals
	\begin{align}
		\int_{C_j(z,r^{\rho})}|\widetilde{K}(x,y)-&\widetilde{K}(x,y)|^2dx \leq \int_{C_j(z,r^{\rho})}{\left( \int_{1}^{\infty}{|\widetilde{K}(x,y,t)-\widetilde{K}(x,z,t)| \frac{dt}{t}} \right)^2dx} \nonumber \\
		& =  \left\{ \left[ \int_{C_j(z,r^{\rho})}{\left( \int_{1}^{\infty}{|\widetilde{K}(x,y,t)-\widetilde{K}(x,z,t)| \frac{dt}{t}} \right)^2dx}  \right]^{\frac{1}{2}} \right\}^{2} \nonumber \\
		& \leq  \left\{ \int_{1}^{\infty}{ \left( \int_{C_j(z,r^{\rho})}{|\widetilde{K}(x,y,t)-\widetilde{K}(x,z,t)|^2dx}  \right)^{\frac{1}{2}}\frac{dt}{t} }\right\}^{2}. \nonumber
	\end{align}
	Let $\Gamma (t)= \| \widetilde{K}(\cdot,y,t)-\widetilde{K}(\cdot,z,t) \|_{L^2[C_j(z,r^{\rho})]}$ and then
	\begin{align} \label{q-hormander_dem_1}
		\left( \int_{C_j(z,r^{\rho})}{|\widetilde{K}(x,y)-\widetilde{K}(x,z)|^2dx} \right)^{\frac{1}{2}} \leq \int_{1}^{r^{-1}}{ \Gamma (t) \frac{dt}{t}} + \int_{r^{-1}}^{\infty}{ \Gamma (t) \frac{dt}{t}}=I_1+I_2.
	\end{align}
	Lets deal first with $I_1$, in which the estimate relies strongly on the assumption $tr<1$. Throughout this proof we consider $N \in \mathbb{Z}_{+}$ a constant that will be chosen conveniently after. Clearly
	\begin{align*}
		\Gamma (t)
		&  \leq \left( \int{|\widetilde{K}(x,y,t)-\widetilde{K}(x,z,t)|^2(1+t^{2\sigma}|x-z|^2)^Ndx}  \right)^{\frac{1}{2}} \\ & \quad \quad \quad  \quad \quad \quad  \times \sup_{x \in C_j(z,r^{\rho})}{(1+t^{2\sigma}|x-z|^2)^{-\frac{N}{2}}}.
	\end{align*}
	We claim that for $m \leq -n[(1-\sigma)/2+\lambda]$
	\begin{align} \label{estimate-hounie-1}
		\left( \int_{\R^n}{|\widetilde{K}(x,y,t)-\widetilde{K}(x,z,t)|^2(1+t^{2\sigma}|x-z|^2)^Ndx}  \right)^{{1}/{2}}  \lesssim  (tr) t^{\frac{\sigma n}{2}+\lfloor \delta \rfloor} \ \ \text{for} \ \  tr \leq 1
	\end{align}
	and  $\displaystyle{	\sup_{x \, \in \, C_j(z,r^{\rho})}{(1+t^{2\sigma}|x-z|^2)^{-{N}/{2}}} \leq \left[ 1+t^{2\sigma}(2^jr^{\rho})^2 \right]^{-{N}/{2}}}$. Using these estimates and the change of variables $\omega=t^{\sigma}2^{j}r^{\rho}$ we obtain
	\begin{align}
		\int_{1}^{r^{-1}}{ \Gamma (t) \frac{dt}{t}} & \lesssim \ \int_{1}^{r^{-1}}{r \, t^{\frac{\sigma n}{2}+\lfloor \delta \rfloor}\left[ 1+t^{2\sigma}(2^jr^{\rho})^2 \right]^{-\frac{N}{2}}dt} \nonumber \\
		&\lesssim r^{1-\frac{\rho n}{2}-\frac{\rho}{\sigma}(1+\lfloor \delta \rfloor)} \, (2^j)^{-\frac{n}{2}-\frac{1+\lfloor \delta \rfloor}{\sigma}} \int_{2^{j}r^{\rho}}^{2^{j}r^{\rho-\sigma}}{\frac{\omega^{\frac{n}{2}-1+\frac{1+\lfloor \delta \rfloor}{\sigma}}}{(1+\omega^2)^{\frac{N}{2}}}d\omega} \nonumber \\
		& \lesssim   \ r^{-\lfloor \delta \rfloor} \, |C_j(z,r^{\rho})|^{-\frac{1}{2}+ \frac{\lfloor \delta \rfloor + 1}{n} \left( \frac{1}{\rho}-\frac{1}{\sigma} \right)} \,\, 2^{-\frac{j}{\rho}(\lfloor \delta \rfloor+1)}, \label{estimation-pseudo}
	\end{align}
	since $\displaystyle{\int_{0}^{\infty}{\frac{\omega^{\frac{n}{2}-1+\frac{1+\lfloor \delta \rfloor}{\sigma}}}{(1+\omega^2)^{\frac{N}{2}}}d\omega} < \infty}$ for $\displaystyle N>\frac{n}{2}+\frac{1+\lfloor \delta \rfloor}{\sigma}$. Lets us give an idea of the proof of \eqref{estimate-hounie-1}. Using integration by parts, for $\alpha \in \Z^{n}_{+}$ such that $|\alpha|\leq N$ we may write
	\begin{align}
		t^{\sigma|\alpha|} &(x-z)^{\alpha}[\widetilde{K}(x,y,t)-\widetilde{K}(y,z,t)] \nonumber \\
		&= \sum_{|\beta|\leq |\alpha|}C_{\alpha, \beta} \, t^{\sigma|\alpha| +\lfloor \delta \rfloor} \int{e^{i(x-z)\cdot \xi} |\xi|^{n(1-\sigma)/2+\sigma|\beta|} \, \partial_{\xi}^{\beta}\left[(e^{i(z-y)\cdot \xi}-1)p(x,\xi)\right]} \nonumber \\
		&\quad \quad \quad \quad \quad \quad \quad  \times	|\xi|^{-n(1-\sigma)/2-\sigma|\beta|} \, \partial_{\xi}^{\alpha-\beta}\left[\psi\left(\frac{|\xi|}{t}\right) \left( \frac{\xi}{t} \right)^{\gamma} \,\, \right] d\xi . \label{estimate-pseudo-kernel-1}
	\end{align} 
	Since $| e^{i(z-y)\cdot \xi}-1 | \leq tr$ and $| \partial_{\xi}^{\beta}e^{i(z-y)\cdot \xi} | \lesssim |\xi|^{-|\beta|} \, (tr)^{|\beta|}$ one can show that if $\chi \in C^{\infty}_{c}(\R^{+})$ is a function such that $\chi =\psi$ on the support of $\psi$, then
	\begin{align*}
		\left\{ |\xi|^{n(1-\sigma)/2+\sigma|\beta|} \partial_{\xi}^{\beta}\left[(e^{i(z-y)\cdot \xi}-1)p(x+z,\xi)\right] \chi \left( {|\xi|}/{t} \right): \ |y-z|<r, \, z \in \R^n   \right\}
	\end{align*}
	is a bounded subset of $S^{m+n(1-\sigma)/2}_{\sigma,b}(\R^n)$ with bounds being less than or equal to $Ctr$. Therefore, since $m\leq -n\lambda$, the family of symbols above defines  pseudodifferential operators that are bounded in $L^2(\R^n)$ with norm proportional to $Ctr$ (see \cite[Theorem 1]{Hounie1986}). Therefore from this consideration and \eqref{estimate-pseudo-kernel-1} we have
	\begin{align*}
		&\left( \int_{\R^n}|\widetilde{K}(x,y,t)-\widetilde{K}(x,z,t)|^2(1+t^{2\sigma}|x-z|^2)^Ndx  \right)^{{1}/{2}}  \\
		& \quad\quad\quad \lesssim tr \sum_{|\alpha|\leq N}\sum_{|\beta| \leq |\alpha|} C_{\alpha} \, \left\|  t^{\sigma|\alpha| + \lfloor \delta \rfloor} \, |\xi|^{-n(1-\sigma)/2-\sigma|\beta|} \partial_{\xi}^{\alpha-\beta}\left[\psi\left(\frac{|\xi|}{t}\right) \left( \frac{\xi}{t} \right)^{\gamma} \,\, \right]  \right\|_{L^2} \\
		&\quad\quad\quad \lesssim (tr)t^{\frac{\sigma n}{2}+\lfloor \delta \rfloor}.
	\end{align*}
	On the other hand, to control $I_2$ we split
	$$\Gamma (t) \leq \| \widetilde{K}(\cdot,y,t) \|_{L^2[C_j(z,r^{\rho})]}+\| \widetilde{K}(\cdot,z,t) \|_{L^2[C_j(z,r^{\rho})]}.
	$$
	If $x \in C_{j}(z,r^{\rho})$ and $|y-z|<r<1$, then $|x-y| \geq |x-z|-|y-z| \geq 2^{j-1}r^{\rho}$	and
	\begin{align*}
		\| \widetilde{K}(\cdot,y,t) \|_{L^2[C_j(z,r^{\rho})]} & \leq \left( \int_{\R^{n}}|\widetilde{K}(x,y,t)|^2(t^{2\sigma}|x-y|^2)^{N}dx \right)^{\frac{1}{2}} \times \\
		& \quad \quad \quad \quad \quad \quad \times \sup_{|x-y|>2^{j-1}r^{\rho}}{(t^{2\sigma}|x-y|^2)^{-N/2}}.
	\end{align*}
	We claim that
	\begin{align} \label{estimate-integral-adj}
		\left(\int_{\R^{n}}{|\widetilde{K}(x,y,t)|^2(t^{2\sigma}|x-y|^2)^{N}dx} \right)^{\frac{1}{2}} \lesssim t^{\frac{\sigma n}{2} + \lfloor \delta \rfloor}
	\end{align}
	and the second term is clearly estimated by $(t^{\sigma}2^{j-1}r^{\rho})^{-N}$. Thus $\| \widetilde{K}(\cdot,y,t) \|_{L^2[C_j(z,r^{\rho})]} \lesssim t^{{\sigma n}/{2} +\lfloor \delta \rfloor}(t^{\sigma}2^{j-1}r^{\rho})^{-N}$. Analogously $ \| \widetilde{K}(\cdot,z,t) \|_{L^2[C_j(z,r^{\rho})]} \lesssim t^{{\sigma n}/{2}+\lfloor \delta \rfloor}(t^{\sigma}2^{j-1}r^{\rho})^{-N}$. Using these estimates and assuming $\displaystyle N>\frac{n}{2}+ \frac{\lfloor \delta \rfloor + 1}{\sigma}$  we obtain
	\begin{align}
		\int_{r^{-1}}^{\infty}{ \Gamma (t) \frac{dt}{t}}  &\lesssim \int_{r^{-1}}^{\infty}{t^{\frac{\sigma n}{2}+\lfloor \delta \rfloor-\sigma N -1}(2^jr^\rho)^{-N}dt}
		\lesssim (2^jr^\rho)^{-\frac{n}{2}-\frac{1+\lfloor \delta \rfloor}{\sigma}} \, r \nonumber \\ 
		& = r^{1-\frac{\rho}{\sigma}-\frac{\rho n}{2}-\frac{\rho \lfloor \delta \rfloor}{\sigma}} (2^j)^{-\frac{n}{2}-\frac{1+\lfloor \delta \rfloor}{\sigma}} \nonumber \\
		& \lesssim r^{-\lfloor \delta \rfloor} \, |C_j(z,r^{\rho})|^{-\frac{1}{2}+ \frac{\lfloor \delta \rfloor + 1}{n} \left( \frac{1}{\rho}-\frac{1}{\sigma} \right)} \,\, 2^{-\frac{j}{\rho}(\lfloor \delta \rfloor+1)}. \label{ineq-1}
	\end{align}
	It just remains to show now \eqref{estimate-integral-adj}. In the same spirit as previously, taking $|\alpha|=N$ we may write
	\begin{align*}
		t^{\sigma |\alpha|}(x-y)\widetilde{K}(x,y,t) \simeq \sum_{|\beta| \leq |\alpha|} t^{\sigma|\alpha| {+ \lfloor \delta \rfloor}} \int{e^{i(x-y)\cdot \xi}\partial_{\xi}^{\beta}p(x,\xi) {\partial_{\xi}^{\alpha-\beta} \left[\psi\left(\frac{|\xi|}{t}\right) \left( \frac{\xi}{t} \right)^{\gamma} \,\, \right]}d\xi}.
	\end{align*}
	Since the class of symbols $\left\{ |\xi|^{n(1-\sigma)/2+\sigma|\beta|}\partial_{\xi}^{\beta}p(x+y,\xi): \ y \in \R^n  \right\}$ are a bounded subset of $S^{m+n(1-\sigma)/2}_{\sigma, b}(\R^n)$, follows directly that the family of pseudodifferential associated  is uniformly bounded on $L^2(\R^n)$. Therefore
	\begin{align*}
		&\left(\int_{\R^n}{|\widetilde{K}(x,y,t)|^2(t^{2\sigma}|x-y|^2)^{N}dx} \right)^{\frac{1}{2}} \\
		& \quad \quad \quad \quad \quad \quad \quad \lesssim \sum_{|\beta|\leq |\alpha|} t^{\sigma|\alpha| {+ \lfloor \delta \rfloor} } \, \left\| |\xi|^{-\frac{n(1-\sigma)}{2}-\sigma|\beta|} {\partial_{\xi}^{\alpha-\beta} \left[\psi\left(\frac{|\xi|}{t}\right) \left( \frac{\xi}{t} \right)^{\gamma} \,\, \right]} \right\|_{L^2} \\
		& \quad \quad \quad \quad \quad \quad \quad \lesssim t^{\frac{\sigma n}{2} + {+ \lfloor \delta \rfloor}}.
	\end{align*}
	Now we are moving on to the case $r>1$. Since we can estimate $\| \widetilde{K}(\cdot,y,t) \|_{L^2[C_j(z,r)]}$ and $\| \widetilde{K}(\cdot,z,t) \|_{L^2[C_j(z,r)]}$ in the same way as before, we obtain for {$$N>\max \left\{ \frac{n}{2}+\frac{\lfloor \delta \rfloor}{\sigma}, \, \frac{n}{2}+ \lfloor \delta \rfloor +1 \right\},$$}
	\begin{align}
		&\left( \int_{C_j(z,r)}{|\widetilde{K}(x,y)-\widetilde{K}(x,z)|^2dx} \right)^{\frac{1}{2}} \nonumber \\
		& \quad \quad \quad \quad \quad \quad \quad  \leq  \int_{1}^{\infty}{ \left( \| \widetilde{K}(\cdot,y,t) \|_{L^2[C_j(z,r)]} + \| \widetilde{K}(\cdot,z,t) \|_{L^2[C_j(z,r)]} \right) \frac{dt}{t}} \nonumber \\
		& \quad \quad \quad \quad \quad \quad \quad \lesssim (2^jr)^{-N} \, \int_{r^{-1}}^{\infty}{t^{\frac{\sigma n}{2} {+\lfloor \delta \rfloor}-\sigma N -1}dt} \nonumber \\
		& \quad \quad \quad \quad \quad \quad \quad \lesssim r^{-\frac{n}{2}-\lfloor \delta \rfloor -(1-\sigma)- \lfloor \delta \rfloor(1-\sigma)} \, (2^j)^{-\frac{n}{2}-\lfloor \delta \rfloor -1} \nonumber \\
		&  \quad \quad \quad \quad \quad \quad \quad \lesssim \ r^{-\lfloor \delta \rfloor} \, |C_j(z,r)|^{-\frac{1}{2}} \,\, 2^{-j(1+\lfloor \delta \rfloor)}. \label{estimate-rgreater1}
	\end{align}
	
	Now we deal with estimates of the adjoint. Suppose first $0<r<1$. Since
	\begin{align*}
		{(-i)^{-|\gamma|}}(2\pi)^{n}[\widetilde{K}(y,x,t)-\widetilde{K}(z,x,t)] &=\int{e^{-i(x-y)\cdot \xi}[p(y,\xi)-p(z,\xi)] \, \psi \left(\frac{|\xi|}{t}\right) \, {\xi^{\gamma}}  \, d\xi} \\
		& \quad \quad + \int{e^{ix \cdot \xi}(e^{iy\xi}-e^{iz \cdot \xi})p(z,\xi) \, \psi\left(\frac{|\xi|}{t}\right) {\xi^{\gamma}}  d\xi}  \\
		&=: f(x-y,y,z,t)+g(x,y,z,t),
	\end{align*}
	then
	\begin{align*}
		&\left( \int_{C_j(z,r^{\rho})}{|\widetilde{K}(y,x,t)-\widetilde{K}(z,x,t)|^{2}dx} \right)^{\frac{1}{2}} \\
		& \quad \quad \quad \quad \quad \quad \quad \lesssim \| g(\cdot,y,z,t) \|_{L^2[C_j(z,r^{\rho})]} + \| f(\cdot-y,y,z,t) \|_{L^2[C_j(z,r^{\rho})]}
	\end{align*}
	and we will obtain analogous estimates for the $L^2-$norm as presented before. Suppose first $tr<1$ and note that
	\begin{align} \label{estimation1-g}
		|g(x,y,z,t)|^2(1+t^{2\sigma}|x|^{2})^{N} = \sum_{|\alpha| \leq N} \left[ |g(x,y,z,t)| \, (t^{\sigma}|x|)^{|\alpha|} \right]^{2}.
	\end{align}
	Considering $G(\xi,y,z,t)=(e^{iy\cdot \xi}-e^{iz\cdot \xi})p(z,\xi) \,\psi \left( {|\xi|}/{t} \right) \, {\xi^{\gamma}}$ and taking the Fourier transform in the first variable we have the identity $\widehat{G}(x,y,z,t)=(2\pi)^{-n} g(x,y,z,t)$. In addition, from mean value inequality it follows for $|y-z|\leq r$ and $tr<1$ that
	\begin{align} \label{estimation2-g}
		\left| \partial_{\xi}^{\beta}[(e^{iy \cdot \xi}-e^{iz\cdot \xi})p(z,\xi)] \right| \lesssim (tr)t^{m-\sigma|\beta|}.
	\end{align}
	Then, from \eqref{estimation1-g} and \eqref{estimation2-g} 
	\begin{align*}
		&\left( \int_{\R^n}{|g(x,y,z,t)|^2(1+t^{2\sigma}|x|^{2})^{N}dx} \right)^{\frac{1}{2}} \lesssim \sum_{|\alpha| \leq N} t^{\sigma |\alpha|} \| \widehat{G}(\cdot,y,z,t) |x|^{|\alpha|} \|_{L^2} \\
		&\quad \quad \lesssim \sum_{|\alpha| \leq N} t^{\sigma |\alpha|} \| \widehat{\partial_{\xi}^{\alpha} G}(\cdot,y,z,t) \|_{L^2} \\
		&\quad \quad \leq \sum_{|\alpha| \leq N}\sum_{|\beta| \leq |\alpha|} t^{\sigma |\alpha| {+ \lfloor \delta \rfloor}} \left\| \partial_{\xi}^{\beta}[(e^{iy \cdot \xi}-e^{iz\cdot \xi})p(z,\xi)] \, \partial_{\xi}^{\alpha-\beta}\left[ \psi\left( \frac{|\xi|}{t} \right) \, {\left(\frac{\xi}{t}\right)^{\gamma}} \,\, \right] \right\|_{L^2} \\
		& \quad \quad \leq \sum_{|\alpha| \leq N}\sum_{|\beta| \leq |\alpha|} t^{\sigma |\alpha| {+\lfloor \delta \rfloor}} \,  (tr) \, t^{m-\sigma |\beta|} \, t^{|\beta|-|\alpha|} \, t^{\frac{n}{2}} \\
		& \quad \quad \lesssim (tr) t^{\frac{\sigma n}{2}{+\lfloor \delta \rfloor}}
	\end{align*}
	since $m \leq -n(1-\sigma)/2$. The estimate for $f$ follows by the same steeps as presented for g. We proceed as before replacing $G$ by $G'(\xi,y,z,t)=[p(y,\xi)-p(z,\xi)] \, \psi \left( |\xi|/t \right) \, {\xi^{\gamma}}$ and using the estimate
	\begin{align*}
		\left| \partial^{\alpha}_{\xi}[p(y,\xi)-p(z,\xi)] \right| \leq (tr) \, t^{m-\sigma |\alpha|},
	\end{align*}
Thus, the conclusion follows in the same way did in \eqref{estimation-pseudo}. If we drop the assumption $tr<1$ we will proceed as following. Write
	\begin{align*}
		g(x,y,z,t) &= \int{e^{-i(x-y)\cdot \xi}p(z,\xi) \, \psi \left( \frac{|\xi|}{t} \right) \, {\xi^{\gamma}} d\xi}-\int{e^{-i(x-z)\cdot \xi}p(z,\xi) \, \psi \left( \frac{|\xi|}{t} \right) \, {\xi^{\gamma}} d\xi} \\
		&=: g_1(x,y,z,t)-g_2(x,y,z,t).
	\end{align*}
	Thus, we will obtain the $L^2-$norm estimate for $g_1$ and $g_2$. In the same way as before
	\begin{align*}
		&\| g_2(\cdot,y,z,t) \|_{L^2[C_j(z,r^{\rho})]}  \\
		& \quad \quad \quad \quad \quad \quad \quad \leq \left( \int{|g_2(x,y,z,t)|^{2} \, [(x-z)^{2}t^{2\sigma}]^{N}dx} \right)^{\frac{1}{2}} 
		\sup_{x \in C_j(z,r^{\rho})}{[(x-z)t^{\sigma}]^{-N}}
	\end{align*}
	and consider $\alpha \in \Z^{n}_{+}$ such that $|\alpha|=N$. Integration by parts gives us
	$$
	g_2(x,y,z,t) \, (x-z)^{\alpha} t^{\sigma |\alpha|} = C \, t^{\sigma |\alpha|} \widehat{\partial_{\xi}^{\alpha}G}(x-z,y,z,t),
	$$
	where $G(\xi,y,z,t)= p(z,\xi) \, \psi \left({|\xi|}/{t} \right)\, {\xi^{\gamma}}$. Using that
	\begin{align*}
		\left| \partial_{\xi}^{\alpha} G(\xi,y,z,t)  \right| &\leq {t^{\lfloor \delta \rfloor}} \sum_{|\beta|\leq|\alpha|}{\left|\partial_{\xi}^{\beta} p(z,\xi) \right| \, \left|\partial_{\xi}^{\alpha - \beta} \left[ \psi \left( \frac{|\xi|}{t} \right) \left( \frac{\xi}{t} \right)^{\gamma} \,\, \right]  \right|} \nonumber \\
		&\leq \sum_{|\beta|\leq|\alpha|} |\xi|^{m-\sigma |\beta|} \, t^{|\beta|-|\alpha|{+\lfloor \delta \rfloor}} \nonumber \\
		&\leq \sum_{|\beta|\leq|\alpha|}{t^{m-\sigma|\alpha|{+\lfloor \delta \rfloor}} \, t^{(1-\sigma)(|\beta|-|\alpha|)}} \lesssim t^{m-\sigma|\alpha|{+\lfloor \delta \rfloor}}  
	\end{align*}
	we get
	\begin{align*}
		\| g_2(\cdot,y,z,t) \, (x-z)^{\alpha}t^{\sigma |\alpha|} \|_{L^2[C_j(z,r^{\rho})]} &\leq  \| t^{\sigma |\alpha|} \widehat{\partial_{\xi}^{\alpha}G}(\cdot-z,y,z,t) \|_{L^2} \\
		&= \| t^{\sigma |\alpha|} \partial_{\xi}^{\alpha}G(\cdot-z,y,z,t) \|_{L^2} \\
		&\lesssim  t^{\sigma |\alpha|} \, t^{m-\sigma |\alpha|+\frac{n}{2} {+\lfloor \delta \rfloor}} \lesssim t^{\frac{\sigma n}{2}{+\lfloor \delta \rfloor}}
	\end{align*}
	since $m \leq -(1-\sigma)n/2$. On the other hand, the same estimate for $g_1$ is valid. Indeed
	\begin{align*}
		&\| g_1(\cdot,y,z,t) \|_{L^2[C_j(z,r^{\rho})]} \\
		& \quad \quad \quad \quad \quad \quad \quad \leq \left( \int{|g_1(x,y,z,t)|^{2} \, [(x-y)^{2}t^{2\sigma}]^{N}dx} \right)^{\frac{1}{2}} 
		\sup_{x \in C_j(z,r^{\rho})}{[(x-y)t^{\sigma}]^{-N}}.
	\end{align*} 
	The control of the integral is analogous as in the previous case and for supremum term note that since $r<1$, $x \in C_j(z,r^{\rho})$ and $|y-z|<r$ we get $|x-y|>2^{j-1} r^{\rho}$ and thus
	\begin{align*}
		\sup_{x \in C_j(z,r^{\rho})}{[(x-y)t^{\sigma}]^{-N}} \leq (2^{j-1}r^{\rho}t^{\sigma})^{-N}.
	\end{align*}
	From that point we proceed as in \eqref{ineq-1} and obtain the desired estimates for $g$. The same argument applies to $f$ if we split
	\begin{align*}
		f(x-y,y,z,t)= & \int{e^{-i(x-y) \cdot \xi}p(y,\xi) \psi \left( \frac{|\xi|}{t} \right) {\, \xi^{\gamma} \,} d\xi} \\
		& \quad  -\int{e^{-i(x-y) \cdot \xi}p(z,\xi) \psi \left( \frac{|\xi|}{t} \right) {\, \xi^{\gamma} \,}d\xi}
	\end{align*}
	and the estimate follows exactly in the same way as did for $g$.
	
	The case $r \geq 1$ is analogous as the previous and we obtain that
	\begin{align*}
		\| g(\cdot, y, z, t) \|_{L^2[C_j(z,r)]} \lesssim t^{\frac{\sigma n}{2} {+\lfloor \delta \rfloor}} \, (2^j r t^{\sigma})^{-N}
	\end{align*}
	and
	\begin{align*}	
	 \| f(\cdot-y, y, z, t) \|_{L^2[C_j(z,r)]} \lesssim t^{\frac{\sigma n}{2} {+\lfloor \delta \rfloor} } \, (2^j r t^{\sigma})^{-N}.
	\end{align*} 
	Thus, the desired estimate follows as in \eqref{estimate-rgreater1}.
\end{proof}

\subsection{Strongly singular $\theta$-Calder\'on-Zygmund operator of type $\sigma$} \label{section-yabuta-operators}

K. Yabuta considered in \cite[Definition 2.1]{Yabuta1985} a generalization of classical Calder\'on-Zygmund operators assuming a $\theta$-modulus of continuity of the kernel 
and a complete study on boundedness ($L^p-L^p$ for $1<p<\infty$, $L^{\infty}-BMO$ and $H^1-L^1$) of this type of operators. Kernels satisfying $\theta$-modulus of continuity are related with general classes of pseudodifferential operators (beyond H\"ormander class), see for instance \cite[Theorems 3.1 and 3.2]{Yabuta1985}.


In this short subsection we introduce a generalization of strongly  singular Calder\'on-Zygmund operators of type $\sigma$ assuming an analogous $\theta$-modulus of continuity of the kernel. This has its own interests and can lead to new paths in connection to pseudo-differential operators associated to rough symbols. 


\begin{definition}
	Let $\theta: (0,\infty) \rightarrow (0, \infty)$ be an increasing function and $0<\sigma \leq 1$. 
	We say that a continuous function $K(x,y)$ defined on $\R^{2n}$ away the diagonal is a $\theta-$kernel of type $\sigma$ if 
	\begin{equation} 
		|K(x,y)| \lesssim \frac{1}{|x-y|^{n}} \quad \, \forall \ x \neq y
	\end{equation}
	and
	\begin{equation} \label{kernel-yabuta}
		|K(x,y)-K(x,z)|+|K(y,x)-K(z,x)| \lesssim \theta \left( \frac{|y-z|}{|x-z|^{\frac{1}{\sigma}}} \right) \, |y-z|^{-n}
	\end{equation}
	for all $|x-z| \geq 2|y-z|^{\sigma}$. A linear and bounded operator $T: \mathcal{S}(\R^{n}) \rightarrow \mathcal{S}'(\R^{n})$ is called a strongly singular $\theta$-Calder\'on-Zygmund operator of type $\sigma$ if it satisfies conditions (i) and (iii) of Theorem \ref{teorema_s-hormander}.
\end{definition}

Clearly if $\theta(t)=t^{\delta}$ for some $0<\delta \leq 1$, then we recover \eqref{pontual_kernel_tipo_sigma} and if one consider the Dini-condition $$\int_{0}^{1} \frac{\theta(t)}{t} dt < \infty$$ then we recover \eqref{hormander} with $\delta=1$.

\begin{theorem}
	Let $0<p\leq 1$ and $T$ a strongly singular $\theta-$Calder\'on-Zygmund operator of type $\sigma$. If $T^{\ast}(x^{\alpha})=0$ for every $|\alpha| \leq \lfloor \delta \rfloor$ and
	\begin{equation} \label{epsilon-yabuta}
		\int_{0}^{1}{\frac{\left[\theta (t)\right]^{s_1}}{t^{1+\delta s_1}}dt} < \infty
	\end{equation}
	for some $\delta>0$ and $1\leq s_1 <\infty$ with $p<s_1$, then $T$ is a bounded operator on $H^p(\R^n)$ to itself for $p_{_0}<p\leq1$, in which $p_{_0}$ is given by \eqref{pcritico}.
\end{theorem}

Conditions like \eqref{epsilon-yabuta} have already been considered in the literature to obtain boundedness of standard $\theta-$Calder\'on-Zygmund operators. For instance, see \cite[Theorem 1.2]{Luong2011}, where the same condition with $\sigma=s_1=1$ has been used in the setting of weighted Hardy spaces and also \cite[Theorems 8 and 9]{YangQuek2000}, for a similar one in weak-Hardy spaces. Conditions like $\int_{0}^{1}\frac{\left[\theta (t)\right]^{a}}{t}dt<\infty$ for $a>0$ have also been considered in the literature (see \cite{LuZhang2014} and their cited papers) and is usually referred as $a-$Dini condition.

\begin{proof}
	Let $a(x)$ be a $(p, \infty)$-atom supported on $B(z,r)$. We show that $Ta$ is a $(p,\rho, q , \lambda,s_2,s_1)-$molecule like in the proof of Theorem \ref{teorema_s-hormander}. Since conditions (M1) and (M3) relies only on the continuity properties of $T$, and not on the regularity of the kernel itself, the proofs will be the same. Suppose first $r>1$ and we show (M2). Since $\theta$ is increasing and by \eqref{kernel-yabuta} it follows that
	\begin{align}
		|Ta(x)| & \leq \int_{B(z,r)}{|K(x,y)-K(x,z)| \ |a(y)|dy} \lesssim r^{\,n\left(1-\frac{1}{p}\right)} \theta \left( \frac{r}{|x-z|^{\frac{1}{\sigma}}} \right)  |x-z|^{-n}. \nonumber
	\end{align}
	Therefore
	\begin{align}
		\int_{\R^n \setminus B(z,2r)}{|Ta(x)|^{s_1}|x-z|^\lambda dx} 
		& \leq r^{\, s_1 n\left(1-\frac{1}{p} \right)} \int_{\R^n \setminus B(z,2r)}{\left[ \theta \left( \frac{(2r)^{\frac{1}{\sigma}}}{|x-z|^{\frac{1}{\sigma}}} \right) \right]^{s_1} |x-z|^{\lambda-s_1n}dx} \nonumber \\
	& = r^{\, \lambda+n\left(1-\frac{s_1}{p}\right)} \int_{|w|>1}{\left[\theta \left( |w|^{-\frac{1}{\sigma}}\right)\right]^{s_1} \, |w|^{\lambda-s_1n}dw} \nonumber \\
		& = r^{\,\lambda+n\left(1-\frac{s_1}{p}\right)} \int_{1}^{\infty}{\left[ \theta \left( u^{-\frac{1}{\sigma}}\right) \right]^{s_1} u^{\,\lambda-n(s_1-1)-1}du} \nonumber \\
		& \lesssim r^{\,\lambda+n\left(1-\frac{2}{p}\right)} \int_{0}^{1}{\frac{\left[\theta(t)\right]^{s_1}}{t^{1+\delta s_1}}t^{-\sigma \lambda+\sigma n (s_1-1)+\delta s_1}dt} \nonumber \\
		& \lesssim r^{\,\lambda+n\left(1-\frac{s_1}{p}\right)}, \nonumber
	\end{align}
	using condition \eqref{epsilon-yabuta} and $\lambda <n(s_1-1)+\frac{\delta s_1}{\sigma}$, which is valid and was already pointed out in \eqref{control-lambda}. For $r<1$,
	\begin{align}
		\int_{\R^n \setminus B(z,2r^{\rho})} &|Ta(x)|^{s_1}|x-z|^\lambda dx \lesssim r^{\,s_1n \left(1-\frac{1}{p} \right)} \int_{\R^n \setminus B(z,2r^{\rho})}{\left[\theta \left( \frac{r}{|x-z|^{\frac{1}{\sigma}}} \right)\right]^{s_1} |x-z|^{\lambda -s_1n}dx} \nonumber \\
		& = r^{\,s_1n\left(1-\frac{1}{p}\right)+\rho\left[\lambda - n(s_1-1)\right]} \int_{|w|>1}{\left[\theta \left( \frac{r^{\,1-\frac{\rho}{\sigma}}}{|w|^{\frac{1}{\sigma}}} \right) \right]^{s_1} |w|^{\lambda-s_1n}dw} \nonumber \\
		& =  r^{\,s_1n\left(1-\frac{1}{p}\right)+\sigma(\lambda-ns_1+n)+\rho-1} \int_{0}^{r^{1-\frac{\rho}{\sigma}}}{\frac{\left[ \theta(t) \right]^{s_1}}{t^{1+\delta s_1}} t^{-\sigma\lambda+\sigma n(s_1-1)+\delta s_1}dt} \nonumber \\
		&  \lesssim r^{\,\rho\lambda + n \left[ \rho \left(1-\frac{s_1}{s_2} \right)+s_1 \left(\frac{1}{q}-\frac{1}{p}\right) \right]}, \nonumber
	\end{align}
where in the last integral we estimate $t \leq r^{1-\frac{\rho}{\sigma}}$ and we choose  $\rho$ as in \eqref{choice-rho}.
\end{proof}

\begin{remark}
\textnormal{Condition \eqref{epsilon-yabuta} can be refined for one related to \eqref{s-hormander-1} and \eqref{s-hormander-2}. Let $I=(2^{-\frac{1}{\sigma}},1)$, $I_{j}^{\rho}=  r^{1-\frac{\rho}{\sigma}} 2^{-\frac{j}{\sigma}} \times I$ and $I_j = r^{1-\frac{1}{\sigma}} 2^{-\frac{j}{\sigma}} \times I$. Then  $\theta-$kernels of type $\sigma$ such that}
	$$	\left(\int_{I_{j}^{\rho}} \frac{[\theta(t)]^{s_1}}{t} dt \right)^{\frac{1}{s_{1}}} \lesssim  | I_{j}^{\rho}|^{\delta} \textnormal{ if } r<1 \textnormal{ and } \left(\int_{I_{j}} \frac{[\theta(t)]^{s_1}}{t} dt \right)^{\frac{1}{s_{1}}} \lesssim (2^{j})^{-\delta} \textnormal{ if } r>1.
	$$	
\textnormal{satisfy the $D_{s_1}$ condition.}
\end{remark}

\addcontentsline{toc}{chapter}{Bibliography}
\bibliographystyle{amsplain}
\bibliography{referencial}

\end{document}